\documentclass[12pt]{amsart}
\usepackage{a4wide,enumerate,xcolor}
\usepackage{amsmath,graphicx}
\allowdisplaybreaks

\let\pa\partial
\let\na\nabla
\let\eps\varepsilon
\newcommand{\N}{{\mathbb N}}
\newcommand{\R}{{\mathbb R}}
\newcommand{\diver}{\operatorname{div}}

\newcommand{\T}{{\mathcal T}}
\newcommand{\E}{{\mathcal E}}
\newcommand{\F}{{\mathcal F}}
\newcommand{\m}{\operatorname{m}}
\newcommand{\dist}{{\operatorname{d}}}
\newcommand{\dd}{\mathrm{d}}

\newtheorem{theorem}{Theorem}
\newtheorem{lemma}[theorem]{Lemma}
\newtheorem{proposition}[theorem]{Proposition}
\newtheorem{remark}[theorem]{Remark}

\begin{document}

\title[BDF2 finite-volume scheme]{A convergent entropy-dissipating BDF2 finite-volume \\
scheme for a population cross-diffusion system }

\author[A. J\"ungel]{Ansgar J\"ungel}
\address{Institute of Analysis and Scientific Computing, Technische Universit\"at Wien,
Wiedner Hauptstra\ss e 8--10, 1040 Wien, Austria}
\email{juengel@tuwien.ac.at}

\author[M. Vetter]{Martin Vetter}
\address{Institute of Analysis and Scientific Computing, Technische Universit\"at Wien,
Wiedner Hauptstra\ss e 8--10, 1040 Wien, Austria}
\email{martin.vetter@tuwien.ac.at}

\date{\today}

\thanks{The authors acknowledge partial support from
the Austrian Science Fund (FWF), grants P33010 and F65.
This work has received funding from the European
Research Council (ERC) under the European Union's Horizon 2020 research and
innovation programme, ERC Advanced Grant no.~101018153.}

\begin{abstract}
A second-order backward differentiation formula (BDF2) finite-volume discretization for a nonlinear
cross-diffusion system arising in population dynamics is studied. The numerical scheme
preserves the Rao entropy structure and conserves the mass. The existence and
uniqueness of discrete solutions and their large-time behavior as well as the convergence of the scheme are proved. The proofs are based on the G-stability of the BDF2 scheme,
which provides an inequality for the quadratic Rao entropy and hence suitable a priori
estimates. The novelty is the extension of this inequality to the system case.
Some numerical experiments in one and two space dimensions underline the theoretical results.
\end{abstract}



\maketitle


\section{Introduction}\label{sec.intro}

The design of structure-preserving finite-volume schemes for parabolic equations 
is fundamental to describe accurately the behavior of the numerical solutions to these equations.
In the literature, usually implicit Euler time discretization
are used to derive such schemes; see, e.g., \cite{ABR11,BCH20,Bes12,CaGa20,JuZu20}.
However, implicit Euler schemes are only first order accurate in time, 
while finite-volume implementations often show second-order accuracy in space 
\cite{CaGa20,JuZu22} (also see \cite{DrNa18} for an analytical result).
In order to match the convergence rates in space and time, there is the need to
design second-order time approximations, which lead to structure-preserving and
convergent schemes. Some works suggest higher-order time discretizations 
(e.g.\ \cite{CaEz17,DWZZ20,Emm09,JuMi15,MaPl19}), 
but they are only concerned with semidiscrete equations or
different numerical methods, or they do not contain any numerical analysis. 
In this paper, we propose a second-order BDF two-point flux approximation
finite-volume scheme, which conserves the mass and
dissipates the Rao entropy, for a nonlinear
cross-diffusion system arising in population dynamics.
The quadratic structure of the Rao entropy allows us to extend the
G-stability theory of Dahlquist to the system case, leading to
existence, uniqueness, and convergence results.

The dynamics of the population density $u_i(x,t)$ of the $i$th species is modeled by the cross-diffusion equation
\begin{equation}\label{1.eq}
  \pa_t u_i = \diver(\gamma\na u_i + u_i\na p_i(u)), \quad 
	p_i(u) := \sum_{j=1}^n a_{ij}u_j\quad\mbox{in }\Omega,\ t>0,\
    i=1,\ldots,n,
\end{equation}
where $\Omega\subset\R^d$ ($d\ge 1$) is a bounded domain and $u=(u_1,\ldots,u_n)$.
This model was derived rigorously from a moderately interacting stochastic particle system 
in a mean-field-type limit \cite{CDJ19}. 
The parameter $\gamma>0$ is related to the stochastic diffusion of the
particle system, and $a_{ij}\in\R$ describes the strength of the repulsive or attractive interaction  
between the $i$th and the $j$th species. We impose initial and no-flux boundary conditions,
\begin{equation}\label{1.bic}
  u_i(0)=u_i^0\quad\mbox{in }\Omega, \quad \na u_i\cdot\nu=0\quad\mbox{on }\pa\Omega,\ t>0,\ i=1,\ldots,n,
\end{equation}
where $\nu$ is the exterior unit normal vector to $\pa\Omega$.
In the absence of the diffusion parameter $\gamma$,
\eqref{1.eq} can be interpreted as a mass conservation equation with the partial
velocity $\na p_i(u)$, which is determined according to Darcy's law by the partial
pressure $p_i(u)$. System \eqref{1.eq} in one space dimension
for two species, $\gamma=0$, and $\det(a_{ij})=0$ was 
first studied in \cite{BGHP85}, proving the global existence of segregated solutions
(i.e., the supports of $u_1$ and $u_2$ do not intersect for all times
if this holds true initially). This result was generalized to arbitrary space dimensions
in \cite{BHIM12}, still for two species.
For an arbitrary number of species, 
the existence of global weak solutions to \eqref{1.eq}--\eqref{1.bic} was shown in \cite[Appendix B]{JPZ22} if $\det(a_{ij})>0$ and the existence of 
local strong solutions was proved in \cite{DHJ22} if $\det(a_{ij})=0$. 

The matrix $A=(a_{ij})\in\R^{n\times n}$ does not need to be symmetric nor positive
definite so that the diffusion matrix associated to \eqref{1.eq} is generally neither
symmetric nor positive definite too. A minimal requirement for local solvability
at the linear level is the parabolicity in the sense of Petrovskii, 
which is satisfied if all eigenvalues of $A$ have a positive real part
\cite{Ama93}. 
Global solvability is guaranteed under the detailed-balance condition, i.e., there
exist $\pi_1,\ldots,\pi_n>0$ such that $\pi_i a_{ij}=\pi_j a_{ji}$ for all $i\neq j$ \cite[Theorem 17]{JPZ22}.
This condition also appears in the theory of time-continuous Markov chains generated
by $A$, and $(\pi_1,\ldots,\pi_n)$ is the associated invariant measure.
We assume this condition throughout this paper.
It implies that $\widetilde{u}_i:=\pi_i u_i$ solves the system
$$
  \pa_t\widetilde{u}_i = \diver\bigg(\widetilde{u}_i
  \sum_{j=1}^n\frac{a_{ij}}{\pi_j} \na\widetilde{u}_j\bigg),
$$
with a symmetric positive definite matrix $(a_{ij}/\pi_j)$. Consequently, 
we may assume, without loss of generality, that the matrix $A$ in
\eqref{1.eq} is already symmetric and positive definite.

Due to the nonlinear cross-diffusion structure, the analysis of \eqref{1.eq} is
highly nontrivial. The key idea of the analysis is to exploit the entropy structure
of \eqref{1.eq}. This means that there exist Lyapunov functionals, called entropies, that are
nonincreasing in time along solutions to \eqref{1.eq}--\eqref{1.bic} and that
provide gradient estimates. In the present situation, these functionals are given by the
Boltzmann (or Shannon) entropy $H_B$ and the Rao entropy $H_R$,
$$
  H_B(u) = \sum_{i=1}^n\int_\Omega u_i(\log u_i-1)\dd x, \quad 
	H_R(u) = \frac12\int_\Omega u^TAu\dd x,
$$
giving formally the entropy equalities
\begin{align}\label{1.Bei}
  \frac{\dd H_B}{\dd t} + \int_\Omega\bigg(4\gamma\sum_{i=1}^n|\na\sqrt{u_i}|^2
	+ \sum_{i,j=1}^n a_{ij}\na u_i\cdot\na u_j\bigg)\dd x &= 0, \\
	\frac{\dd H_R}{\dd t} + \int_\Omega\bigg(\gamma\sum_{i,j=1}^n
    a_{ij}\na u_i\cdot\na u_j
	+ \sum_{i=1}^n u_i|\na p_i(u)|^2\bigg)\dd x &= 0, \label{1.Rei}
\end{align}
and thus providing gradient bounds for $u_i$. 
The Boltzmann entropy is related to the thermodynamic entropy of the system, while
the Rao entropy measures the functional diversity of the species \cite{Rao82}.

Since the Boltzmann entropy $H_R$ is convex, the implicit Euler scheme preserves the
entropy inequality \eqref{1.Bei} (see, e.g., \cite{JuZu22} for a related system).
The logarithmic structure of $H_R$ seems to prevent
entropy stability in higher-order schemes like BDF or Crank--Nicolson approximations
\cite{GuSh20}. However, thanks to the quadratic structure of the Rao entropy $H_R$, 
we are able to prove stability of $H_R$ for the BFD2 approximation.
To explain the idea, let $\T$ be a triangulation of $\Omega$ into control volumes
$K\subset\Omega$ with measure $\m(K)$ and let $\Delta t$ be the time step size. 
Furthermore, let $u_{i,K}^k$ be an approximation of $u_i(x_K,t_k)$, where
$x_K\in K$ and $t_k=k\Delta t$.
We write the BDF2 discretization of \eqref{1.eq} as
\begin{equation}\label{1.num}
  \frac{\m(K)}{\Delta t}\bigg(\frac32 u_{i,K}^k - 2u_{i,K}^k + \frac12 u_{i,K}^{k-2}\bigg)
	+ \sum_{\sigma\in\E_K}\F_{i,K,\sigma}^k = 0,
\end{equation}
where $\E_K$ is the set of the edges (or faces) of $K$ and $\F_{i,K,\sigma}^k$ is the numerical
flux, defined in \eqref{2.flux} below. The usual idea to derive a priori bounds
is to choose the test function $u_{i,K}^k$ in \eqref{1.num} and to use the inequality
\begin{equation}\label{1.gstab}
  \bigg(\frac32 u_{i,K}^k - 2u_{i,K}^k + \frac12 u_{i,K}^{k-2}\bigg)u_{i,K}^k 
	\ge h_0(u_{i,K}^k,u_{i,K}^{k-1}) - h_0(u_{i,K}^{k-1},u_{i,K}^{k-2}),
\end{equation}
where
$$
	h_0(a,b) = \frac14\big(5a^2 - 4ab
	+ b^2\big) = \frac14\begin{pmatrix} a \\ b \end{pmatrix}^T
	\begin{pmatrix} 5 & -2 \\ -2 & 1 \end{pmatrix}
	\begin{pmatrix} a \\ b\end{pmatrix}, \quad a,b\in\R,
$$
is a positive definite quadratic form.
Assuming that $\F_{i,K,\sigma}^k u_{i,K}^k$ can be bounded from below, this gives a priori 
bounds for $(u_{i,K}^k)^2$. Inequality \eqref{1.gstab} can be explained in the
framework of Dahlquist's G-stability theory \cite{Hil97}.

In our case, we need the test function $p_i(u_K^k)$ to derive the discrete
analog of \eqref{1.Rei}. Then the question is whether there exists a functional
$h(u,v)$ such that
\begin{equation}\label{1.gstab2}
  \sum_{i=1}^n\bigg(\frac32 u_{i,K}^k - 2u_{i,K}^k + \frac12 u_{i,K}^{k-2}\bigg)p_i(u_K^k)
	\ge h(u_{K}^k,u_{K}^{k-1}) - h(u_{K}^{k-1},u_{K}^{k-2}).
\end{equation}
Note that we need to sum over all species in this inequality.
The main novelty of this paper is the observation that the scalar inequality
\eqref{1.gstab} can be extended to inequality \eqref{1.gstab2} for vectors 
$u$, $v\in\R^n$.
Indeed, we show in Lemma \ref{lem.bdf2} that \eqref{1.gstab2} holds for 
\begin{equation}\label{1.h}
  h(u,v) = \frac14(5u^TAu - 4u^TAv + v^TAv)
	= \frac14\begin{pmatrix} u \\ v\end{pmatrix}^T
	\begin{pmatrix} 5A & -2A \\ -2A & A \end{pmatrix}
	\begin{pmatrix} u \\ v\end{pmatrix}
\end{equation}
with $u,v\in\R^n$. Introducing the discrete Rao entropy by 
$H(u,v)=\sum_{K\in\T}\m(K)h(u,v)$ for piecewise
constant functions $u$ and $v$, this yields the
BDF2 analog of the Rao entropy inequality
$$
  H(u^k,u^{k-1}) + c\Delta t|u^k|_{1,2,\T}^2 \le H(u^{k-1},u^{k-2})\quad\mbox{for }k\ge 2,
$$
where $|\cdot|_{1,2,\T}$ is the discrete $H^1(\Omega)$ norm, defined in Section 
\ref{sec.func}, and $c>0$ depends on the smallest eigenvalue of $A$ and
on $\gamma$. This inequality is the key for proving our main results:
\begin{itemize}
\item Existence and uniqueness of discrete solutions: 
There exists a solution $u_i^k$ to the BDF2 finite-volume scheme
\eqref{1.num}, which
conserves the mass $\sum_{K\in\T}\m(K)u_{i,K}^k$ of the $i$th species and
dissipates the discrete Rao entropy. Moreover, the solution
is unique if $\Delta t/(\Delta x)^{d+2}$ is sufficiently small,
where $\Delta x$ is the size of the mesh (Theorem \ref{thm.ex}).
This unusual quotient comes from an inverse inequality needed to bound 
higher-order norms.
\item Large-time behavior: The discrete solution $u_i^k$ converges for large times $k\to\infty$ to the constant
steady state $\bar{u}_i=\m(\Omega)^{-1}\int_\Omega u_i^0\dd x$
with a quasi-explicit exponential rate (Theorem \ref{thm.large}).
The proof uses the well-established relative entropy (or energy) method, but the two-step scheme requires an iteration of this argument.
\item Convergence of the discrete scheme: 
The fully discrete solution converges to a solution to the semidiscrete problem
if $\Delta x\to 0$, and the semidiscrete solution converges to a weak
(nonnegative) solution to \eqref{1.eq}--\eqref{1.bic} as $\Delta t\to 0$ (up to 
subsequences; see Theorem \ref{thm.conv}).
\item Convergence rate: If the solution to \eqref{1.eq}--\eqref{1.bic}
is sufficiently smooth, the semidiscrete solution converges
with order two, as expected for the BDF2 scheme (Theorem \ref{thm.second}).
\end{itemize}

The paper is organized as follows. The numerical scheme and our main results are detailed in Section \ref{sec.scheme}. In Section \ref{sec.ex}, we prove the existence
and uniqueness of a discrete solution, while its large-time behavior is analyzed in
Section \ref{sec.large}. Section \ref{sec.conv} is devoted to the convergence of the full scheme, and the second-order convergence in time is
verified in Section \ref{sec.second}. 
Finally, we present in Section \ref{sec.num} some
numerical examples in one and two space dimensions.


\section{Numerical scheme and main results}\label{sec.scheme}

We need some simple auxiliary results and some notation before formulating the numerical scheme and the main results.

\subsection{Some linear algebra}

We denote by $|\cdot|$ the Euclidean norm on $\R^n$. Given a symmetric positive matrix
$A\in\R^{n\times n}$, we introduce the weighted norm $|u|_A^2:=u^TAu$ and the 
weighted inner product $(u,v)_A:=u^TAv$ for $u,v\in\R^n$. With this notation, the 
discrete Rao entropy density can be written as
\begin{equation}\label{2.h}
  h(u,v) = \frac14(5|u|_A^2 - 4(u,v)_A + |v|_A^2) \quad\mbox{for }u,v\in\R^n.
\end{equation}
Denoting by $\lambda_m>0$ the
smallest and by $\lambda_M>0$ the largest eigenvalue of $A$, it holds that
\begin{equation}\label{2.normineq}
  \lambda_m|u|^2 \le |u|_A^2 \le \lambda_M|u|^2 \quad\mbox{for }u\in\R^n.
\end{equation}
Let $\lambda_1,\ldots,\lambda_n>0$ be the eigenvalues of $A$. Then
the eigenvalues of the matrix in \eqref{1.h} equal $(3\pm\sqrt{8})\lambda_i$ for
$i=1,\ldots,n$. This shows that for $u,v\in\R^n$,
\begin{equation}\label{2.hineq}
\begin{aligned}
  \frac14(3-\sqrt{8})(|u|_A^2+|v|_A^2) &\le h(u,v) 
	\le \frac14(3+\sqrt{8})(|u|_A^2+|v|_A^2), \\
  \frac14(3-\sqrt{8})\lambda_m(|u|^2+|v|^2) &\le h(u,v) 
	\le \frac14(3+\sqrt{8})\lambda_M(|u|^2+|v|^2).
\end{aligned}
\end{equation}

\subsection{Spatial domain and mesh}

Let $d\ge 1$ and let $\Omega\subset\R^d$ be a bounded polygonal (if $d=2$) or polyhedral 
(if $d\ge 3$) domain.
We associate to this domain an admissible mesh, given by
(i) a family $\T$ of open polygonal or polyhedral 
control volumes, which are also called cells,  
(ii) a family $\E$ of edges (or faces if $d\ge 3$),
and (iii) a family of points $(x_K)_{K\in\T}$ associated to the control volumes
and satisfying \cite[Definition 9.1]{EGH00}. This definition implies that
the straight line $\overline{x_Kx_L}$ between two centers of neighboring cells
is orthogonal to the edge (or face) $\sigma=K|L$ between two cells.
For instance, triangular meshes with acute angles, Delaunay meshes, rectangular
meshes, and Vorono\"{\i} meshes satisfy this condition \cite[Example 9.2]{EGH00}.
The size of the mesh is given by $\Delta x = \max_{K\in\T}\operatorname{diam}(K)$.
The family of edges $\E$ is assumed to consist of interior edges
$\E_{\rm int}$ satisfying $\sigma\subset\Omega$ and boundary edges $\sigma\in\E_{\rm ext}$
satisfying $\sigma\subset\pa\Omega$. For a given $K\in\T$, $\E_K$ denotes the 
set of edges of $K$ with $\E_K= \E_{{\rm int},K}\cup\E_{{\rm ext},K}$.
For any $\sigma\in\E$, there exists at least one cell $K\in\T$ such that
$\sigma\in\E_K$. 

For given $\sigma\in\E$, we define the distance
$$
  \dist_\sigma = \begin{cases}
	\dist(x_K,x_L) &\quad\mbox{if }\sigma=K|L\in\E_{{\rm int},K}, \\
	\dist(x_K,\sigma) &\quad\mbox{if }\sigma\in\E_{{\rm ext},K},
	\end{cases}
$$
where d is the Euclidean distance in $\R^d$, and the transmissibility coefficient
\begin{equation}\label{2.trans}
  \tau_\sigma = \frac{\m(\sigma)}{\dist_\sigma},
\end{equation}
where $\m(\sigma)$ denotes the $(d-1)$-dimensional 
Lebesgue measure of $\sigma$. 
We suppose the following mesh regularity condition: There exists
$0<\zeta\le 1/2$ such that for all $K\in\T$ and $\sigma\in\E_K$,
\begin{equation}\label{2.regul}
  \dist(x_K,\sigma)\ge \zeta \dist_\sigma.
\end{equation}
This is equivalent to 
$$
  \eta\le\frac{\dist(x_K,\sigma)}{\dist(x_L,\sigma)}\le \frac{1}{\eta}
	\quad\mbox{for all }\sigma=K|L,
$$
where $\eta=\zeta/(1-\zeta)\in(0,1]$. The statement follows by observing that
$\dist(x_K,\sigma)+\dist(x_L,\sigma)=\dist(x_K,x_L)$ holds, which is a consequence
of the orthogonality of $\sigma=K|L$ and $\overline{x_Kx_L}$.
Hence, the mesh regularity \eqref{2.regul} means that the mesh is locally quasi-uniform.
A consequence of the mesh regularity is the following estimate
\begin{equation}\label{2.sum}
  \sum_{\sigma\in\E_{{\rm int},K}}\m(\sigma)\dist_\sigma
  \le \frac{1}{\zeta}\sum_{\sigma\in\E_{{\rm int},K}}\m(\sigma)
  \dist(x_K,\sigma)	= \frac{d}{\zeta}\m(K) 
  \quad\mbox{for }K\in\T,
\end{equation}
where we used in the last step the formula for the volume of a 
(hyper-)pyramid.

\subsection{Function spaces}\label{sec.func}

Given a triangulation $\T$, let $T>0$, $N_T\in\N$ and introduce the time step size $\Delta t=T/N_T$
and the time steps $t_k=k\Delta t$ for $k=0,\ldots,N_T$. 
We set $\Omega_T=\Omega\times(0,T)$.
The space of piecewise constant functions is defined by
$$
  V_\T = \bigg\{v: \Omega\to\R:\exists (v_K)_{K\in\T}\subset\R,\
	v(x)=\sum_{K\in\T}v_K\mathbf{1}_K(x)\bigg\},
$$
where $\mathbf{1}_K$ is the indicator function on $K$. To define a norm
on this space, we define for $K\in\T$, $\sigma\in\E_K$,
$$
  v_{K,\sigma} = \begin{cases}
	v_L &\quad\mbox{if }\sigma=K|L\in\E_{{\rm int},K}, \\
	v_K &\quad\mbox{if }\sigma\in\E_{{\rm ext},K},
	\end{cases} \quad
	\textrm{D}_{K,\sigma} v := v_{K,\sigma}-v_K, \quad 
	\textrm{D}_\sigma v := |\mathrm{D}_{K,\sigma} v|.
$$
Let $1\le q<\infty$ and $v\in V_\T$.
The discrete $W^{1,q}(\Omega)$ norm on $V_\T$ is given by 
\begin{align*}
  & \|v\|_{1,q,\T} = \big(\|v\|_{0,q,\T}^q+|v|_{1,q,\T}^q\big)^{1/q}, 
	\quad\mbox{where} \\
	& \|v\|_{0,q,\T}^q = \sum_{K\in\T}\m(K)|v_K|^q, \quad
  |v|_{1,q,\T}^q = \sum_{\sigma\in\E_{\rm int}}\m(\sigma)\dist_\sigma
	\bigg|\frac{\text{D}_{\sigma} v}{\dist_\sigma}\bigg|^q\quad\mbox{for }v\in V_\T.
\end{align*}
When $q=\infty$, we define $|v|_{1,\infty,\T}=\max_{\sigma\in\E_{\rm int}}
|\mathrm{D}_\sigma v|/\dist_\sigma$.
If $v=(v_1,\ldots,v_n)\in V_\T^n$ is a vector-valued function, 
we write for notational convenience
$$
  \|v\|_{0,q,\T} = \sum_{i=1}^n\|v_i\|_{0,q,\T}, \quad
	\|\na v\|_{0,q,\T} = \sum_{i=1}^n\|\na v_i\|_{0,q,\T}.
$$
We associate to the discrete $W^{1,q}$ norm a dual norm with respect 
to the $L^2$ inner product:
$$
  \|v\|_{-1,q,\T} = \sup\bigg\{\int_\Omega vw \dd x: w\in V_\T,\
	\|w\|_{1,q,\T}=1\bigg\}.
$$

Finally, we introduce the space $V_{\T,\Delta t}$ of piecewise constant
functions with values in $V_\T$,
\begin{equation*}
  V_{\T,\Delta t} = \bigg\{v:\Omega_T\to\R:
	\exists(v^k)_{k=1,\ldots,N_T}\subset V_\T,\
	v(x,t) = \sum_{k=1}^{N_T} v^k(x)\mathbf{1}_{[t_{k-1},t_k)}(t)\bigg\},
\end{equation*}
equipped with the $L^2(0,T;H^1(\Omega))$ norm 
$$
  \|v\|_{L^2(0,T;H^1(\Omega))} 
	= \left(\sum_{k=1}^{N_T}\Delta t \|v^k\|_{1,2,\T}^{2}\right)^{1/2} \quad 
	\mbox{for all }v\in V_{\T,\Delta t}.
$$ 

\subsection{Discrete gradient}

The discrete gradient is defined on a dual mesh. For this, we 
define the cell $T_{K,\sigma}$ of the dual mesh for $K\in\T$ and $\sigma\in\E_K$:
\begin{itemize}
\item ``Diamond'': Let $\sigma=K|L\in \E_{{\rm int},K}$. Then $T_{K,\sigma}$ is that cell
whose vertices are given by $x_K$, $x_L$, and the end points of the edge $\sigma$. In higher dimensions, they might be (double) (hyper-)pyramids.
\item ``Triangle'': Let $\sigma\in\E_{{\rm ext},K}$. Then $T_{K,\sigma}$ is that cell
whose vertices are given by $x_K$ and the end points of the edge $\sigma$.
\end{itemize}
The union of all ``diamonds'' and ``triangles'' $T_{K,\sigma}$ equals the domain $\Omega$
(up to a set of measure zero).
The property that the straight line $\overline{x_Kx_L}$ is orthogonal to the edge 
$\sigma=K|L$ implies that
\begin{equation*}
  \m(\sigma)\dist(x_K,x_L)=d\m(T_{K,\sigma})\quad\mbox{for all }
	\sigma=K|L\in\E_{\rm int}.
\end{equation*}
The approximate gradient of $v\in V_{\T,\Delta t}$ is then defined by
$$
  \na^{\T} v(x,t) = \frac{\m(\sigma)}{\m(T_{K,\sigma})}
	\mathrm{D}_{K,\sigma}(v^k)\nu_{K,\sigma}
	\quad\mbox{for }x\in T_{K,\sigma},\ t\in(t_{k-1},t_k],
$$
where $\nu_{K,\sigma}$ is the unit vector that is normal to $\sigma$ and
points outwards of $K$. 

\subsection{Numerical scheme}

The initial functions are approximated by their $L^2(\Omega)$-orthogo\-nal projection
on $V_\T$:
\begin{equation}\label{2.init}
  u_{i,K}^0 = \frac{1}{\m(K)}\int_K u_i^0(x)\dd x \quad\mbox{for all } K\in\T,\ 
	i=0,\ldots,n.
\end{equation}
Let $u_K^{k-1}=(u_{1,K}^{k-1},\ldots,u_{n,K}^{k-1})$ 
for $K \in \T$ be given. Since the BDF2 scheme is a two-step method, we need a first
time step which is computed from the implicit Euler method. The following
time steps are determined from the BDF2 method. The finite-volume scheme reads as
\begin{align}
  \frac{\m(K)}{\Delta t}(u_{i,K}^1-u_{i,K}^0) 
	+ \sum_{\sigma\in\E_K}\F_{i,K,\sigma}^1 &= 0, \label{2.sch1} \\
	\frac{\m(K)}{\Delta t}\bigg(\frac32 u_{i,K}^k - 2u_{i,K}^{k-1} + \frac12 u_{i,K}^{k-2}\bigg)
	+ \sum_{\sigma\in\E_K}\F_{i,K,\sigma}^k &= 0, \quad k\ge 2, \label{2.sch2}
\end{align}
for $i=1,\ldots,n$, $K\in\T$, and the numerical fluxes are given by
\begin{equation}\label{2.flux}
  \F_{i,K,\sigma}^k = -\tau_\sigma\big(\gamma\textrm{D}_{K,\sigma}u_i^k
	+ (u_{i,\sigma}^k)^+\textrm{D}_{K,\sigma} p_i(u^k)\big),
\end{equation}
where $\tau_\sigma$ is defined in \eqref{2.trans} and $z^+=\max\{0,z\}$ denotes the positive part of $z\in\R$.
Finally, the so-called mobility is given by
\begin{equation}\label{2.mean}
  u_{i,\sigma}^k = M(u_{i,K}^k,u_{i,L}^k)\quad\mbox{for }\sigma=K|L,
  \quad u_{i,\sigma}^k = 0\quad\mbox{else}, 
\end{equation} 
where $M$ is a general mean function
satisfying
\begin{itemize}
\item[(i)] $M:[0,\infty)^2\to[0,\infty)$ is Lipschitz continuous, satisfies
$M(u,u)=u$ (consistency), and has linear growth in the sense $M(u,v)\le |u|+|v|$ for $u,v\ge 0$.
\item[(ii)] There exists $C>0$ such that $|M(u,v)-u|\le C|u-v|$
for all $u,v\ge 0$.
\end{itemize}
Examples for $M$ are $M(u,v)=(u+v)/2$ or $M(u,v)=\max\{u,v\}$.
Note that we do not need logarithmic mean functions like in \cite{JuZu22},
since we do not use the chain rule in the cross-diffusion part, so that
we can use simpler expressions.

\begin{remark}[Nonnegativity]\rm
We truncate the mobility by $(u_{i,\sigma}^k)^+$ in the numerical flux
\eqref{2.flux} to ensure the discrete Rao entropy inequality (see \eqref{2.ei.bdf} below).
Indeed, when testing \eqref{2.sch2} with $p_i(u^k)$, we need that the sum
$\sum_{\sigma\in\E_{\rm int}}\tau_\sigma(u_{i,\sigma}^k)^+|\mathrm{D}_{K,\sigma}p_i(u^k)|^2$
is nonnegative. Unfortunately, the quadratic
Rao entropy does not allow us to prove the nonnegativity of the discrete solution, and standard maximum principle arguments do not apply here,
so that the truncation cannot be removed.
A positivity-preserving BDF2 finite-difference scheme was proposed
in \cite{CWWW19}, but the proof relies on discrete $L^\infty(\Omega)$ bounds
for $u^{k-1}$, which are not available in our case. 
Also the Shannon entropy does not help (as in \cite{JuZu22}), 
since it is not compatible with the BDF2 discretization.
Indeed, when we wish to derive a discrete analog of \eqref{1.Bei}, 
we need a finite continuous functional $h(u,v)$ satisfying
$h(u,u)=\sum_{i=1}^n u_i(\log u_i-1)$ (consistency condition) such that
$$
  \sum_{i=1}^n\bigg(\frac32 u_{i,K}^k - 2u_{i,K}^{k-1}
  + \frac12 u_{i,K}^{k-2}\bigg)\log u_{i,K}^k 
  \ge h(u_{K}^k,u_K^{k-1}) - h(u_K^{k-1},u_K^{k-2}). 
$$
If $u_{i,K}^k=u_{i,K}^{k-1}\to 0$ and $u_{i,K}^{k-2}>0$
for all $i\in\{1,\ldots,n\}$, the previous
inequality converges to $-\infty\ge -h(0,u_K^{k-2})$, which is absurd.
At least, we obtain nonnegative solutions in the limit 
$(\Delta x,\Delta t)\to 0$; see Theorem \ref{thm.conv} below.
\end{remark}

\begin{remark}[Discrete integration by parts]\rm
The fluxes $\mathcal{F}_{i,K,\sigma}^k$ are consistent approximations
of the exact fluxes through the edges if we impose the conservation
$\F_{i,K,\sigma}+\F_{i,L,\sigma}=0$ for all edges $\sigma=K|L$, requiring that
they vanish on the Neumann boundary edges, i.e., $\F_{i,K,\sigma}=0$ for all
$\sigma\in\E_{{\rm ext},K}$. In particular, for $v=(v_K)\in V_\T$,
the following discrete integration-by-parts formulas hold:
\begin{equation}\label{2.dibp}
  \sum_{K\in\T}\sum_{\sigma\in\E_K}\F_{i,K,\sigma} v_K 
	= -\sum_{\substack{\sigma\in\E_{\rm int} \\ \sigma=K|L}}\F_{i,K,\sigma}
	\mathrm{D}_{K,\sigma}v, \quad
	\sum_{K\in\T}\sum_{\sigma\in\E_K}\tau_\sigma(\mathrm{D}_{K,\sigma}v)v_K
	= -|v|_{1,2,\T}^2.
\end{equation}
\end{remark}

\subsection{Main results}\label{sec.main}

We impose the following hypotheses.

\begin{itemize} 
\item[(H1)] Data: $\Omega\subset\R^d$ with $d\ge 1$ is a bounded polygonal 
($d=2$) or polyhedral ($d\ge 3$) domain, $T>0$, and $u^0\in L^2(\Omega;\R^n)$.
We set $\Omega_T=\Omega\times(0,T)$.
\item[(H2)] Discretization: $\mathcal{T}$ is an admissible discretization of 
$\Omega$ satisfying \eqref{2.regul} and $t_k=k\Delta t$ for $k=1,\ldots,N_T$.
\item[(H3)] Coefficients: Let $\gamma>0$, and
$A=(a_{ij})\in\R^{n\times n}$ is symmetric and positive definite.
Let $\lambda_m>0$ and $\lambda_M>0$ be the smallest and largest eigenvalue of $A$, respectively.
\end{itemize}

The positivity of $\gamma$ is not needed for the existence analysis but for
the convergence result, where we need higher-order integrability that
is deduced via the discrete Gagliardo--Nirenberg inequality from the
gradient bound. As mentioned in the introduction, the symmetry and
positive definiteness of $A$ can be replaced by the positivity of the
real parts of the eigenvalues of $A$ and the detailed-balance condition.

Recall the discrete BDF2 Rao entropy (see \eqref{2.h})
$$
  H(u,v) = \sum_{K\in\T}\m(K)h(u_K,v_K) 
	= \frac14\sum_{K\in\T}\m(K)\big(5|u_K|_A^2 - 4(u_K,v_K)_A + |v_K|_A^2\big)
$$
for $u,v\in V_\T$. If $u=v$, this expression reduces to the usual discrete Rao entropy, used
for the implicit Euler scheme, $H(u):=H(u,u)=\frac12\sum_{K\in\T}\m(K)|u_K|_A^2$.
Our first result is the existence of a discrete solution.

\begin{theorem}[Existence and uniqueness of discrete solutions]\label{thm.ex}
Let Hypotheses (H1)--(H3) hold, let $k\in\N$, and let $u^{k-1}\in V_\T^n$ be given. 
Then there exists a solution $u^k=(u_1^k,\ldots,u_n^k)
\in V_\T^n$ to scheme \eqref{2.init}--\eqref{2.mean} satisfying
the discrete entropy inequality
\begin{align}\label{2.ei.bdf}
  H(u^k,u^{k-1}) + \gamma\Delta t|A^{1/2}u^k|_{1,2,\T}^2 &\le H(u^{k-1},u^{k-2})
	\quad\mbox{for } k\ge 2, \\
  H(u^1) + \gamma\Delta t|A^{1/2}u^1|_{1,2,\T}^2 &\le H(u^0), \label{2.ei.euler}
\end{align}
and the scheme preserves the mass,
$\sum_{K\in\T}\m(K)u_{i,K}^k=\int_\Omega u_i^0(x)\dd x$ for $i=1,\ldots,n$, $k\ge 1$.
These results also hold if $\gamma=0$.
Furthermore, the solution is unique if $\gamma>0$, 
$\min_{\sigma\in\E_{\rm int}}\dist_\sigma\ge\xi\Delta x$ 
for some $\xi>0$, and
$$
  \frac{\Delta t}{(\Delta x)^{d+2}} 
  < \frac{C(d,\xi,\zeta)\gamma\lambda_m^2}{\lambda_M^2L^2 H(u^0)},
$$ 
where $\zeta$ is defined in \eqref{2.regul} and $L$ is the Lipschitz constant of the mean function $M$, defined in \eqref{2.mean}.
\end{theorem}

The existence of a discrete solution is proved by a fixed-point argument using 
the Brouwer degree theorem.
Uniform estimates are obtained from the discrete Rao entropy inequality
\eqref{2.ei.bdf}, where the BDF2 time approximation is estimated according to \eqref{1.gstab2}.
This inequality, which is the key of our analysis, is proved in Lemma \ref{lem.bdf2}.
 
The uniqueness of solutions
is proved by using the relative entropy method, which
is equivalent to the energy method in the present case, since the Rao entropy is quadratic.
In other words, we use the test function $p_i(u^k)-p_i(v^k)$ in the difference
of the equations \eqref{2.sch2} satisfied by two discrete solutions $u^k$ and $v^k$. The cross-diffusion part contains cubic expressions, which turn into quadratic ones if $|A^{1/2}v^k|_{1,\infty,\T}$ is bounded (similar as in
\cite{ChJu19}). By an inverse inequality, this norm is bounded, 
up to some factor, 
by $(\Delta x)^{-d/2-1}\|A^{1/2}v^k\|_{0,2,\T}$, and $\|A^{1/2}v^k\|_{0,2,\T}$ is bounded because of \eqref{2.ei.bdf}--\eqref{2.ei.euler}. The remaining quadratic expression
is estimated by using the gradient bounds (which requires $\gamma>0$) and
the discrete $L^2(\Omega)$ bound coming from the time discretization
(and introducing the factor $\Delta t$). The condition 
$\min_{\sigma\in\E_{\rm int}}\dist_\sigma\le\xi\Delta x$ is discussed in 
Remark \ref{rem.xi}.

For the next result, 
we set $\bar{u}_i=\m(\Omega)^{-1}\int_\Omega u_i^0\dd x$ and recall
the discrete Poincar\'e--Wirtinger inequality 
$\|v-\bar{v}\|_{0,2,\T}\le C_P\zeta^{-1/2}|v|_{1,2,\T}$ for $v\in V_\T$
\cite[Theorem 3.6]{BCF15}. Then, in view of \eqref{2.normineq},
\begin{equation}\label{2.PWI}
  \|A^{1/2}(v-\bar{v})\|_{0,2,\T}\le C_P\bigg(\frac{\lambda_M}{\lambda_m\zeta}\bigg)^{1/2}
	|A^{1/2}v|_{1,2,\T} \quad\mbox{for }v\in V_\T.
\end{equation}

\begin{theorem}[Large-time behavior]\label{thm.large}
Let $u^k$ be a solution to scheme \eqref{2.init}--\eqref{2.mean}. Then, for $k\ge 2$,
$$
  \|A^{1/2}(u^k-\bar{u})\|_{0,2,\T} \le 
	\sqrt{2}\|A^{1/2}(u^0-\bar{u})\|_{0,2,\T}
    (1+\kappa\Delta t)^{-(k-2)/4},
$$
where $\kappa=4\gamma\lambda_m\zeta/((3+\sqrt{8})C_P^2\lambda_M)$
and $C_P>0$ is the constant of the Poincar\'e--Wirtinger inequality \eqref{2.PWI}. 
\end{theorem}

The theorem states that $u^k$ converges exponentially fast to the constant steady state $\bar{u}$.
Indeed, setting $\lambda_{\Delta t}:=\log(1+\kappa\Delta t)/(\Delta t)
\nearrow\kappa$ as $\Delta t\to 0$, we have
$$
  \|A^{1/2}(u^k-\bar{u})\|_{0,2,\T} \le \sqrt{2}
  \|A^{1/2}(u^0-\bar{u})\|_{0,2,\T}\exp(-\lambda_{\Delta t}t_k), \quad k\ge 2.
$$
The proof of Theorem \ref{thm.large} is based on the discrete entropy inequality
for the discrete relative Rao entropy $H(u^k-\bar{u},u^{k-1}-\bar{u})$, 
similar to \eqref{2.ei.bdf}. 
Indeed, by the discrete Poincar\'e--Wirtinger inequality, the discrete gradient term
is bounded from below by the discrete $L^2(\Omega)$ norm of $u^k-\bar{u}$. As
$H(u^k-\bar{u},u^{k-1}-\bar{u})$ can be estimated in terms of the discrete $L^2(\Omega)$
norms of $u^k-\bar{u}$ and $u^{k-1}-\bar{u}$, we need to iterate the entropy inequality
a second time. Then, using \eqref{2.hineq}, we arrive at the inequality
$$
  H(u^k-\bar{u},u^{k-1}-\bar{u}) \le (1+\kappa\Delta t)^{-1}
	H(u^{k-2}-\bar{u},u^{k-3}-\bar{u}),
$$
and solving this recursion shows the result.

The numerial convergence of the scheme is proved in two steps. First, we show that the
fully discrete solution $u_m^k\in V_\T^n$, indexed with the space grid size $\Delta x_m\to 0$
as $m\to\infty$, converges, up to a subsequence, to a solution $u^k\in H^1(\Omega)$ to the semidiscrete system
\begin{equation*}
\begin{aligned}
  \frac{1}{\Delta t}(u_i^1-u_i^0) &= \diver(\gamma_i\na u_i^1 + u_i^1\na p_i(u^1)), \\
  \frac{1}{\Delta t}\bigg(\frac32 u_i^k - 2u_i^{k-1} + \frac12 u_i^{k-2}\bigg)
	&= \diver(\gamma_i\na u_i^k + (u_i^k)^+\na p_i(u^k))\quad\mbox{in }\Omega,
\end{aligned}
\end{equation*}
with no-flux boundary conditions $\na u_i^k\cdot\nu=0$ on $\pa\Omega$, $i=1,\ldots,n$.
Second, we prove that a subsequence of the sequence of semidiscrete solutions
converges to a weak solution to \eqref{1.eq}--\eqref{1.bic} as $\Delta t\to 0$. Both steps may be summarized as follows (the precise
convergence statements can be found in Propositions \ref{prop.space} and \ref{prop.time}).

\begin{theorem}[Convergence of the scheme]\label{thm.conv}
Let Hypotheses (H1)--(H3) hold and let $(\T_m)_{m\in\N}$ be a sequence of admissible
discretizations of $\Omega$ satisfying \eqref{2.regul} uniformly in $m$ and
$\Delta x_m\to 0$, $\Delta t_m\to 0$ as $m\to\infty$.
Then the solution $(u_m)$ to \eqref{2.init}--\eqref{2.mean}, constructed in Theorem
\ref{thm.ex}, converges, up to a subsequence, as $m\to\infty$ to a function $u=(u_1,\ldots,u_n)$ 
satisfying $u_i\ge 0$ in $\Omega_T$ for $i=1,\ldots,n$, 
$u_i\in L^2(0,T;H^1(\Omega))$,
$\pa_t u_i\in L^{2d+4}(0,T;W^{1,2d+4}(\Omega)')$, and $u$
is a weak solution to \eqref{1.eq}--\eqref{1.bic}.
\end{theorem}

The proof is based on suitable estimates uniform with respect to $\Delta x_m$ and $\Delta t_m$,
derived from the entropy inequality \eqref{2.ei.bdf}. For the limit $\Delta x_m\to 0$, we follow the strategy of \cite{CLP03}. 
The compactness argument is different, since we still keep the time discretization. 
The limit $\Delta t_m\to 0$ is based on a higher-order integrability
property derived from the Gagliardo--Nirenberg inequality and on the Aubin--Lions compactness lemma in the version of \cite{DrJu12}. 

We need the condition $\gamma>0$ since the application of the discrete Gagliardo--Nirenberg inequality requires 
discrete gradient bounds. However, the term involving $p_i(u)$ only provides 
a bound for the discrete kinetic energy 
$\sum_{\sigma\in\E_{\rm int}}\tau_\sigma
(u_{i,\sigma}^k)^+|\mathrm{D}_{K,\sigma}p_i(u^k)|^2$, from which we are unable
to conclude gradient bounds. For the Euler scheme, this issue can be overcome by using the Boltzmann entropy inequality, which provides bounds in 
$L^2(0,T;H^1(\Omega))$ and $L^\infty(0,T;L^1(\Omega))$ (see \eqref{1.Bei}),
and consequently in $L^{2+2/d}(\Omega_T)$, which is the required
higher-order integrability bound. As mentioned in
the introduction, this entropy is not compatible with the BDF2 
discretization. Therefore, the restriction $\gamma>0$ seems to
be unavoidable with our approach.

Finally, we verify that the convergence of the semidiscrete system is of second order. 

\begin{theorem}[Second-order convergence]\label{thm.second}
Let $u^k$ be a solution to \eqref{5.BDF2} and assume that the 
solution to \eqref{1.eq}--\eqref{1.bic} satisfies $u\in C^3([0,T];L^2(\Omega))\cap L^\infty(0,T;W^{1,\infty}(\Omega))$.
Furthermore, let $\eps>0$ be arbitrary and assume that
$$
  \Delta t < \frac{4(3-\sqrt{8})\gamma\lambda_m}{\lambda_M^3
	\|\na u\|_{L^\infty(\Omega_T)}^2+4\gamma\lambda_m\eps}.
$$
Then there exists $C(\eps)>0$, which is of order $\eps^{-1/2}$ as $\eps\to 0$ but independent of $\Delta t$, such that
$$
  \max_{k=1,\ldots,N_T}\|A^{1/2}(u_i^k-u_i(t_k))\|_{L^2(\Omega)}
	\le C(\eps)(\Delta t)^2\quad\mbox{for }i=1,\ldots,n.
$$
\end{theorem}

We allow for the parameter $\eps>0$ to minimize the time step size
constraint; however, optimizing this constraint gives large constants
$C(\eps)$.
The theorem is proved by analyzing the relative entropy $H(u(t_k)-u^k,u(t_{k-1})-u^{k-1})$, 
using a Taylor expansion for $u_i$ up to order $(\Delta t)^3$ (which requires a bound for $\pa_t^3 u_i$), and iterating the entropy inequality once more. The resulting recursive inequality for the relative entropy can be solved, leading to the desired second-order bound.


\section{Proof of Theorem \ref{thm.ex}}\label{sec.ex}

First, we make precise inequality \eqref{1.gstab2}. Recall definition \eqref{1.h} of $h(u,v)$
and let $H(u,v)=\sum_{K\in\T}\m(K)h(u,v)$ be the discrete Rao entropy.

\begin{lemma}[BDF2 inequality]\label{lem.bdf2}
It holds for $u,v,w\in\R^n$ that
$$
  \bigg(\frac32u-2v+\frac12w\bigg)^TAu
	= h(u,v) - h(v,w) + \frac14|u-2v+w|_A^2.
$$
In particular, for $u^k$, $u^{k-1}$, $u^{k-2}\in V_\T^n$,
$$
  \sum_{i,j=1}^n\sum_{K\in\T}\m(K)
	\bigg(\frac32 u_{i,K}^k - 2u_{i,K}^{k-1} 
    + \frac12 u_{i,K}^{k-2}\bigg)a_{ij} u_{j,K}^k
	\ge H(u^k,u^{k-1})-H(u^{k-1},u^{k-2}).
$$
\end{lemma}

\begin{proof}
The proof follows by a direct computation.
\end{proof}


\subsection{Definition and continuity of the fixed-point operator}

We assume that $k\ge 2$, since the existence of a solution $u^1\in V_\T^n$ to 
the Euler scheme \eqref{2.sch2} satisfying \eqref{2.ei.euler}
follows from \cite[Theorem 1]{JuZu20}.
Let $u^{k-1}\in V_\T^n$ be given and let $R>0$, $\delta>0$. We set
$$
  Z_R = \big\{w=(w_1,\ldots,w_n)\in V_\T^n: \|w_i\|_{1,2,\T}<R
	\mbox{ for }i=1,\ldots,n\big\},
$$
and let $w\in Z_R$. We consider the linear regularized problem
\begin{equation}\label{3.lin}
  \eps\bigg(\sum_{\sigma\in\E_K}\tau_\sigma\mathrm{D}_{K,\sigma} w_i^\eps
	- \m(K)w_{i,K}^\eps\bigg) = \frac{\m(K)}{\Delta t}
	\bigg(\frac32 w_{i,K}-2u_{i,K}^{k-1}+\frac12u_{i,K}^{k-2}\bigg)
	+ \sum_{\sigma\in\E_K}\F^+_{i,K,\sigma}(w)
\end{equation}
for $i=1,\ldots,n$, $K\in\T$, where
$$
  \F^+_{i,K,\sigma}(w) = -\tau_\sigma\big(\gamma\mathrm{D}_{K,\sigma}w_i
	+ w_{i,\sigma}^+\mathrm{D}_{K,\sigma} p_i(w)\big).
$$ 
The $\eps$-regularization guarantees the coercivity of the associated
bilinear form, while the truncation 
$w_{i,\sigma}^+$ is needed to obtain the nonnegativity
of the entropy dissipation (see the estimate of $I_6$ below).

We claim that \eqref{3.lin} has a unique solution $w^\eps\in V_\T^n$.
Indeed, since the mapping $g(w^\eps)=\eps(\sum_{\sigma\in\E_K}\tau_\sigma
\mathrm{D}_{K,\sigma} w_i^\eps - \m(K)w_{i,K}^\eps)$ is linear and acting on
finite-dimensional spaces, we only need to verify its injectivity. Let $w^\eps$ be in the kernel of this mapping. Multiplying $g(w^\eps)=0$ by $w_{i,K}^\eps$, summing over $K\in\T$, and using the discrete integration-by-parts formula \eqref{2.dibp}
gives
$$
  0 = \sum_{K\in\T}\sum_{\sigma\in\E_K}\tau_\sigma(\mathrm{D}_{K,\sigma}w_i^\eps)
	w_{i,K}^\eps - \sum_{K\in\T}\m(K)(w_{i,K}^\eps)^2 = -\|w_i^\eps\|_{1,2,\T}^2.
$$
This yields $w^\eps=0$ and proves the claim.

Next, we show that the fixed-point mapping $F:Z_R\to V_\T^n$, $F(w)=w^\eps$, is continuous.
For this, we multiply \eqref{3.lin} by $-w_{i,K}^\eps$, sum over $K\in\T$, and use
discrete integration by parts and the Cauchy--Schwarz inequality:
\begin{align}
  \eps\|w_i^\eps\|_{1,2,\T}^2 &= -\frac{1}{\Delta t}\sum_{K\in\T}\m(K)\bigg(\frac32 w_{i,K}
	- 2u_{i,K}^{k-1} + \frac12 u_{i,K}^{k-2}\bigg)w_{i,K}^\eps
	+ \sum_{\substack{\sigma\in\E_{\rm int} \\ \sigma=K|L}}\F_{i,K,\sigma}^+(w)
	\mathrm{D}_{K,\sigma}w_i^\eps \nonumber \\
	&\le \frac{1}{\Delta t}\bigg\|\frac32 w_i - 2u_i^{k-1} + \frac12 u_i^{k-2}\bigg\|_{0,2,\T}
	\|w_i^\eps\|_{0,2,\T} + \gamma|w_i|_{1,2,\T}|w_i^\eps|_{1,2,\T} 
    \label{3.auxlin} \\
	&\phantom{xx}{}- \sum_{\substack{\sigma\in\E_{\rm int} \\ \sigma=K|L}}
	\tau_\sigma(w_{i,\sigma})^+\mathrm{D}_{K,\sigma}
	p_i(w)\mathrm{D}_{K,\sigma}w_i^\eps. \nonumber
\end{align}
For the last term, we use the Cauchy--Schwarz inequality and the fact that
any norm on $V_\T^n$ is equivalent:
\begin{align*}
  -\sum_{\substack{\sigma\in\E_{\rm int} \\ \sigma=K|L}} &
	\tau_\sigma(w_{i,\sigma})^+\mathrm{D}_{K,\sigma}p_i(w)\mathrm{D}_{K,\sigma}w_i^\eps
	= -\sum_{j=1}^n\sum_{\substack{\sigma\in\E_{\rm int} \\ \sigma=K|L}}\tau_\sigma 
	a_{ij}w_{i,\sigma}^+\mathrm{D}_{K,\sigma}w_i\mathrm{D}_{K,\sigma}w_i^\eps \\
	&\le \sum_{j=1}^n\bigg(\sum_{\substack{\sigma\in\E_{\rm int} \\ \sigma=K|L}}
	\tau_\sigma|\mathrm{D}_\sigma w_i^\eps|^2
	\bigg)^{1/2}\bigg(\sum_{\substack{\sigma\in\E_{\rm int} \\ \sigma=K|L}}\tau_\sigma 
	a_{ij}^2(w_{i,\sigma}^+)^2|\mathrm{D}_{\sigma}w_j|^2\bigg)^{1/2} \\
	&\le C(A)\|w\|_{0,\infty,\T}\sum_{j=1}^n|w_i^\eps|_{1,2,\T}|w_j|_{1,2,\T} 
	\le C(A,R)\|w_i^\eps\|_{1,2,\T},
\end{align*}
where we took into account the linear growth of $w_{i,\sigma}^+$ with respect to $w_{i,K}$ and $w_{i,L}$ (see \eqref{2.mean}) and the definition of $Z_R$.
Inserting these estimates into \eqref{3.auxlin} and dividing by $\|w_i^\eps\|_{1,2,\T}$,
it follows that $\eps\|w_i^\eps\|_{1,2,\T}\le C(A,R)$. 

This bound allows us to verify the continuity of $F$. 
Indeed, let $w^\ell\to w$ as $\ell\to\infty$ and
set $w^{\eps,\ell}=F(w^\ell)$. Then $(w^{\eps,\ell})_{\ell\in\N}$ is uniformly
bounded in the discrete $H^1(\Omega)$ norm. Therefore, there exists a subsequence,
which is not relabeled, such that $w^{\eps,\ell}\to w^\eps$ as $\ell\to\infty$.
Passing to the limit $\ell\to\infty$ in scheme \eqref{3.lin}, we see that
$w^\eps$ is a solution of the scheme and consequently $w^\eps=F(w)$. Since the solution to the
linear scheme \eqref{3.lin} is unique, the entire sequence $(w^{\eps,\ell})_{\ell\in\N}$
converges to $w^\eps$, which shows the continuity of $F$. 

\subsection{Existence of a fixed point}

According to the Brouwer degree fixed-point theorem, it is sufficient to show that
for all $(w^\eps,\rho)\in\overline{Z}_R\times[0,1]$ such that $w^\eps=\rho F(w^\eps)$,
it holds that $w^\eps\not\in\pa Z_R$ or, equivalently, $\|w^\eps\|_{1,2,\T}<R$. 
We claim that this is true for sufficiently large $R>0$. Indeed, let $w^\eps$ be
such a fixed point. It satisfies
\begin{align*}
  \eps\bigg(\sum_{\sigma\in\E_K}&\tau_\sigma\mathrm{D}_{K,\sigma} w_i^\eps - \m(K)w_{i,K}^\eps\bigg) \\
  &= \frac{\rho}{\Delta t}\m(K)\bigg(\frac32 w_{i,K}^\eps - 2u_{i,K}^{k-1}
	+ \frac12 u_{i,K}^{k-2}\bigg) + \rho\sum_{\sigma\in\E_K}\F_{i,K,\sigma}^+(w^\eps).
\end{align*}
We multiply this equation by $-(\Delta t) p_i(w^\eps)$ and sum over $i=1,\ldots,n$, $K\in\T$. Then $0=I_1+I_2+I_3$, where
\begin{align*}
  I_1 &= -\eps\Delta t\sum_{i,j=1}^n\sum_{K\in\T}
  \bigg(\sum_{\sigma\in\E_K}\tau_\sigma
  \mathrm{D}_{K,\sigma} w_i^\eps 
  - \m(K)w_{i,K}^\eps\bigg)a_{ij}w_{j,K}^\eps, \\
  I_2 &= \rho\sum_{i,j=1}^n\sum_{K\in\T}
	a_{ij}\bigg(\frac32 w_{i,K}^\eps - 2u_{i,K}^{k-1} + \frac12 u_{i,K}^{k-2}\bigg)w_{j,K}^\eps, \\
  I_3 &= -\rho\Delta t\sum_{i=1}^n\sum_{\substack{\sigma\in\E_{\rm int} \\ \sigma=K|L}}\F_{i,K,\sigma}^+(w^\eps)\mathrm{D}_{K,\sigma}p_i(w^\eps).
\end{align*}
By discrete integration by parts,
\begin{align*}
  I_1 &= \eps\Delta t\sum_{i,j=1}^n\bigg(\sum_{\substack{\sigma\in\E_{\rm int} \\ \sigma=K|L}}\tau_\sigma a_{ij}\mathrm{D}_{K,\sigma}w_i^\eps
  \mathrm{D}_{K,\sigma}w_j^\eps + \sum_{K\in\T}\m(K)a_{ij} w_{i,K}^\eps
  w_{j,K}^\eps\bigg) \\
  &\ge \eps\lambda_m\Delta t(|w^\eps|_{1,2,\T}^2+\|w^\eps\|_{0,2,\T}^2)
  = \eps\lambda_m\Delta t\|w^\eps\|_{1,2,\T}^2,
\end{align*}
and by Lemma \ref{lem.bdf2},
$$
  I_2 \ge H(w^\eps,u^{k-1}) - H(u^{k-1},u^{k-2}).
$$
For the third term, we obtain
\begin{align*}
  I_3 &= \rho\Delta t\sum_{i,j=1}^n\sum_{\substack{\sigma\in\E_{\rm int} \\ \sigma=K|L}}\tau_\sigma\gamma
	a_{ij}\mathrm{D}_{K,\sigma}w_i^\eps\mathrm{D}_{K,\sigma}w_j^\eps
	+ \rho\Delta t\sum_{i=1}^n\sum_{\substack{\sigma\in\E_{\rm int} \\ \sigma=K|L}}
	\tau_\sigma(w_{i,\sigma}^\eps)^+
	\bigg(\sum_{j=1}^n a_{ij}\mathrm{D}_{K,\sigma}w_j^\eps\bigg)^2 \\
	&\ge \gamma\rho\Delta t
	\sum_{\substack{\sigma\in\E_{\rm int} \\ \sigma=K|L}}\tau_\sigma
	|A^{1/2}\mathrm{D}_{K,\sigma}w^\eps|^2 
	= \gamma\rho\Delta t|A^{1/2}w^\eps|_{1,2,\T}^2.
\end{align*}
Collecting these estimates gives
\begin{equation}\label{3.rho}
  \eps\Delta t\|w^\eps\|_{1,2,\T}^2 + H(w^\eps,u^{k-1}) 
	+ \gamma\Delta t\rho|A^{1/2}w^\eps|_{1,2,\T}^2 \le H(u^{k-1},u^{k-2}).
\end{equation}
Setting $R=(\eps\Delta t)^{-1/2}H(u^{k-1},u^{k-2})^{1/2}+1$,
we infer that $\|w^\eps\|_{1,2,\T}^2\le (R-1)^2<R^2$ and thus $w^\eps\not\in\pa Z_R$,
which shows the claim. Hence, there exists a fixed point $w^\eps$ to $F$, 
which is a solution to
\begin{equation}\label{3.nonlin}
  \eps\bigg(\sum_{\sigma\in\E_K}\tau_\sigma\mathrm{D}_{K,\sigma} w_i^\eps
	- \m(K)w_{i,K}^\eps\bigg) = \frac{\m(K)}{\Delta t}
	\bigg(\frac32 w_{i,K}^\eps-2u_{i,K}+\frac12u_{i,K}^{k-2}\bigg)
	+ \sum_{\sigma\in\E_K}\F_{i,K,\sigma}(w^\eps).
\end{equation}

\subsection{Limit $\eps\to 0$}

The solution $w^\eps$ to \eqref{3.nonlin} satisfies the regularized entropy inequality \eqref{3.rho} with $\rho=1$, and the right-hand side is independent of $\eps$ and $M$. It follows from the
Bolzano--Weierstra{\ss} theorem that there exists a
subsequence of $w^\eps$, which is not relabeled, such that $w^\eps\to w$
as $\eps\to 0$. In particular, $\eps^{1/2} w^\eps\to 0$.
Since the problem is finite dimensional, we can pass to the limit $\eps\to 0$
in \eqref{3.nonlin}. Consequently, $u^k:=w$ is a solution to \eqref{2.sch2}--\eqref{2.mean}.
The same limit in \eqref{3.rho} with $\rho=1$ leads to the discrete entropy inequality
of Theorem \ref{thm.ex}, which finishes the proof.


\subsection{Uniqueness of solutions}

Let $u^k,v^k\in V_\T^n$ be two solutions to \eqref{2.init}--\eqref{2.mean}
with the same initial data $u^0=v^0$.
We take the difference of the equations satisfied by $u^k$ and $v^k$, multiply the resulting equation by
$p_i(u_K^k)-p_i(v_K^k)=\sum_{j=1}^n a_{ij}(u_{j,K}^k-v_{j,K}^k)$, 
sum over $i=1,\ldots,n$, $K\in\T$,
and use discrete integration by parts. This leads to
$0=I_4+I_5+I_6$, where
\begin{align*}
  I_4 &= \frac{3}{2\Delta t}\sum_{i,j=1}^n\sum_{K\in\T}\m(K)
    a_{ij}(u_{i,K}^k-v_{i,K}^k)
	(u_{j,K}^k-v_{j,K}^k) \\
  I_5 &= \sum_{i,j=1}^n
	\sum_{\substack{\sigma\in\E_{\rm int} \\ \sigma=K|L}}\tau_\sigma
	\gamma a_{ij}\mathrm{D}_{K,\sigma}(u_i^k-v_i^k)\mathrm{D}_{K,\sigma}
    (u_j^k-v_j^k) \\
  I_6 &= \sum_{i,j,\ell=1}^n
	\sum_{\substack{\sigma\in\E_{\rm int} \\ \sigma=K|L}}\tau_\sigma
	a_{ij}\big((u_{i,\sigma}^k)^+\mathrm{D}_{K,\sigma}u_j^k 
	- (v_{i,\sigma}^k)^+\mathrm{D}_{K,\sigma}v_j^k\big)
	a_{i\ell}\mathrm{D}_{K,\sigma}(u_\ell^k-v_\ell^k).
\end{align*}
By the definition of the weighted norm,
$I_4 = (3/(2\Delta t))\|A^{1/2}(u^k-v^k)\|_{0,2,\T}^2$.
Furthermore,
\begin{align*}
  I_5 \ge \gamma\sum_{\substack{\sigma\in\E_{\rm int} \\ \sigma=K|L}}
  \tau_\sigma|(A^{1/2}\mathrm{D}_{K,\sigma}(u^k-v^k))|^2
  = \gamma|A^{1/2}(u^k-v^k)|_{1,2,\T}^2.
\end{align*}
We add and subtract the term $(u_{i,\sigma}^k)^+\mathrm{D}_{K,\sigma}v_j^k$ in $I_6$ and apply the Cauchy--Schwarz inequality:
\begin{align*}
  I_6 &= \sum_{i,j,\ell=1}^n
	\sum_{\substack{\sigma\in\E_{\rm int} \\ \sigma=K|L}}\tau_\sigma
	 a_{ij}a_{i\ell}\Big((u_{i,\sigma}^k)^+\mathrm{D}_{K,\sigma}(u_j^k-v_j^k) \\
	&\phantom{xx}{}+ \big((u_{i,\sigma}^k)^+ - (v_{i,\sigma}^k)^+\big)
	\mathrm{D}_{K,\sigma}v_j^k\Big)\mathrm{D}_{K,\sigma}(u_\ell^k-v_\ell^k) \\
	&= \sum_{i=1}^n\sum_{\substack{\sigma\in\E_{\rm int} \\ \sigma=K|L}}\tau_\sigma
	(u_{i,\sigma}^k)^+\bigg(\sum_{j=1}^n a_{ij}\mathrm{D}_{K,\sigma}(u_j^k-v_j^k)\bigg)
	\bigg(\sum_{\ell=1}^n a_{i\ell}\mathrm{D}_{K,\sigma}(u_\ell^k-v_\ell^k)\bigg) \\
	&\phantom{xx}{}- \sum_{\substack{\sigma\in\E_{\rm int} \\ \sigma=K|L}}\tau_\sigma
	(A\mathrm{D}_{K,\sigma}v^k)^T\big[\operatorname{diag}\big((u_{i,\sigma}^k)^+
	- (v_{i,\sigma}^k)^+\big)A^{1/2}\big](A^{1/2}\mathrm{D}_{K,\sigma}(u^k-v^k)) \\
	&\ge -\bigg(\sum_{\substack{\sigma\in\E_{\rm int} \\ \sigma=K|L}}\tau_\sigma
	|A\mathrm{D}_{K,\sigma}v^k|^2\big|\operatorname{diag}\big((u_{i,\sigma}^k)^+
	- (v_{i,\sigma}^k)^+\big)A^{1/2}\big|^2\bigg)^{1/2} \\
	&\phantom{xx}{}\times\bigg(\sum_{\substack{\sigma\in\E_{\rm int} \\ \sigma=K|L}}\tau_\sigma
	|A^{1/2}\mathrm{D}_{K,\sigma}(u^k-v^k)|^2\bigg)^{1/2},
\end{align*}
where $\operatorname{diag}((u_{i,\sigma}^k)^+ - (v_{i,\sigma}^k)^+)$ denotes the diagonal matrix with the entries $(u_{i,\sigma}^k)^+ - (v_{i,\sigma}^k)^+$ for $i=1,\ldots,n$.
Together with
\begin{align*}
  & |A\mathrm{D}_{K,\sigma}v^k| \le |A^{1/2}||A^{1/2}\mathrm{D}_{K,\sigma}v^k|
	\le \lambda_M^{1/2}|\mathrm{D}_{K,\sigma}v^k|_A\quad\mbox{and} \\
  & \big|\operatorname{diag}\big((u_{i,\sigma}^k)^+	- (v_{i,\sigma}^k)^+\big)A^{1/2}\big|
	\le \big|\operatorname{diag}\big((u_{i,\sigma}^k)^+	- (v_{i,\sigma}^k)^+\big)\big||A^{1/2}| \\ 
	&\phantom{xx}\le \lambda_M^{1/2}\max_{i=1,\ldots,n}|(u_{i,\sigma}^k)^+ - (v_{i,\sigma}^k)^+|
	\le \lambda_M^{1/2}\max_{i=1,\ldots,n}|u_{i,\sigma}^k - v_{i,\sigma}^k|,
\end{align*}
we find that
$$
  I_6 \ge -\lambda_M|A^{1/2}v^k|_{1,\infty,\T}\max_{i=1,\ldots,n}\bigg(
	\sum_{\substack{\sigma\in\E_{\rm int} \\ \sigma=K|L}}\m(\sigma)\dist_\sigma
	|u_{i,\sigma}^k-v_{i,\sigma}^k|^2\bigg)^{1/2}|A^{1/2}(u^k-v^k)|_{1,2,\T}.
$$

It remains to estimate the term involving the difference $|u_{i,\sigma}^k-v_{i,\sigma}^k|$.
By the Lipschitz continuity of the mean function $M(u_{i,K}^k,u_{i,L}^k)
= u_{i,\sigma}^k$ with Lipschitz constant $L>0$ and the mesh regularity \eqref{2.sum},
\begin{align*}
  \sum_{\substack{\sigma\in\E_{\rm int} \\ \sigma=K|L}} & \m(\sigma)\dist_\sigma
	|u_{i,\sigma}^k-v_{i,\sigma}^k|^2
	= \frac12\sum_{K\in\T}\sum_{\sigma\in\E_{{\rm int},K}}\m(\sigma)\dist_\sigma
	|u_{i,\sigma}^k-v_{i,\sigma}^k|^2 \\
	&\le \frac{L^2}{2}\sum_{K\in\T}\sum_{\substack{\sigma\in\E_{\rm int} \\ \sigma=K|L}}\m(\sigma)\dist_\sigma
	\big(|u_{i,K}^k-v_{i,K}^k|+|u_{i,L}^k-v_{i,L}^k|\big)^2 \\
	&\le 2L^2\sum_{K\in\T}\sum_{\substack{\sigma\in\E_{\rm int} \\ \sigma=K|L}}\m(\sigma)\dist_\sigma|u_{i,K}^k-v_{i,K}^k|^2
	\le \frac{2dL^2}{\zeta}\sum_{K\in\T}\m(K)|u_{i,K}^k-v_{i,K}^k|^2 \\
	&= \frac{2dL^2}{\zeta}\|u^k-v^k\|_{0,2,\T}^2
	\le \frac{2dL^2}{\lambda_m\zeta}\|A^{1/2}(u^k-v^k)\|_{0,2,\T}^2,
\end{align*}
This shows that
$$
  I_6 \ge -\frac{\lambda_M L}{\lambda_m^{1/2}}
    \bigg(\frac{2d}{\zeta}\bigg)^{1/2} |A^{1/2}v^k|_{1,\infty,\T}
	\|A^{1/2}(u^k-v^k)\|_{0,2,\T}|A^{1/2}(u^k-v^k)|_{1,2,\T}.
$$
Collecting the estimates for $I_4$, $I_5$, and $I_6$  
and using Young's inequality gives
\begin{align}\label{3.aux}
  \frac{3}{2\Delta t}&\|A^{1/2}(u^k-v^k)\|_{0,2,\T}^2
	+ \gamma|A^{1/2}(u^k-v^k)|_{1,2,\T}^2 \\
	&\le \frac{\lambda_ML}{\lambda_m^{1/2}}
	\bigg(\frac{2d}{\zeta}\bigg)^{1/2}|A^{1/2}v^k|_{1,\infty,\T}
	\|A^{1/2}(u^k-v^k)\|_{0,2,\T}|A^{1/2}(u^k-v^k)|_{1,2,\T} \nonumber \\
    &\le \frac{3}{2\Delta t}\|A^{1/2}(u^k-v^k)\|_{0,2,\T}^2
    + \frac{\Delta t}{3}\frac{d\lambda_M^2L^2}{\lambda_m\zeta}
    |A^{1/2}v^k|_{1,\infty,\T}^2|A^{1/2}(u^k-v^k)|_{1,2,\T}^2. \nonumber
\end{align}
Now, the inverse inequality $|A^{1/2}v^k|_{1,\infty,\T}\le C'(d)
(\Delta x)^{-d/2}\zeta^{-1/2} |A^{1/2}v^k|_{1,2,\T}$ 
\cite[Prop. 3.10]{DFGH03} and condition $\dist_\sigma\ge\xi\Delta x$
imply that
\begin{align*}
  |A^{1/2}v^k|_{1,\infty,\T}^2
  &\le \frac{C'(d)^2}{(\Delta x)^d\zeta}
  \sum_{\substack{\sigma\in\E_{\rm int} \\ \sigma=K|L}}
  \frac{\m(\sigma)}{\dist_\sigma}
  |\mathrm{D}_{K,\sigma}(A^{1/2}v^k)|^2 \\
  &\le \frac{C'(d)^2}{(\Delta x)^d\zeta}
  \sum_{\substack{\sigma\in\E_{\rm int} \\ \sigma=K|L}}
  \frac{\m(\sigma)\dist_\sigma}{\xi^2(\Delta x)^2}
  |\mathrm{D}_{K,\sigma}(A^{1/2}v^k)|^2.
\end{align*}
It follows from \eqref{2.sum} and 
$|\mathrm{D}_{K,\sigma}(A^{1/2}v^k)|^2\le 2(|v_K^k|_A^2+|v_L^k|_A^2)$ that
\begin{align*}
  |A^{1/2}v^k|_{1,\infty,\T}^2
  &\le \frac{C'(d)^2}{(\Delta x)^{d+2}\xi^2\zeta}
  \sum_{K\in\T}\frac{d}{\zeta}\m(K)|\mathrm{D}_{K,\sigma}(A^{1/2}v^k)|^2 \\
  &\le \frac{2dC'(d)^2}{(\Delta x)^{d+2}(\xi\zeta)^2}\sum_{K\in\T}\m(K)
  |A^{1/2}v^k_K|^2 = \frac{C(d,\xi,\zeta)}{(\Delta x)^{d+2}}
  \|A^{1/2}v^k\|_{0,2,\T}^2.
\end{align*}
Using this inequality as well as the bound
$$
  (3-\sqrt{8})\|A^{1/2}v^k\|_{0,2,\T}^2
  \le 4H(v^1,v^0)\le 2(3+\sqrt{8})(H(v^1)+H(v^0))
  \le 4(3+\sqrt{8})H(u^0),
$$
obtained from \eqref{2.ei.bdf}--\eqref{2.ei.euler}, we deduce from 
\eqref{3.aux}, for another constant $C(d,\xi,\zeta)$ that
$$
  \gamma|A^{1/2}(u^k-v^k)|_{1,2,\T}^2
  \le C(d,\xi,\zeta)\frac{\lambda_M^{2}L^2}{\lambda_m}
  \frac{\Delta t}{(\Delta x)^{d+2}}H(u^0)|A^{1/2}(u^k-v^k)|_{1,2,\T}^2.
$$
Then our smallness condition on $\Delta t/(\Delta x)^{d+2}$ implies that
$|A^{1/2}(u^k-v^k)|_{1,2,\T}=0$ and consequently, $u^k=v^k$,
finishing the proof.

\begin{remark}\rm\label{rem.xi}
The quasi-uniform condition $\min_{\sigma\in\E_{\rm int}}\dist_\sigma
\ge\xi\Delta x>0$ implies condition (23) in \cite{EGH99},
since the mesh regularity \eqref{2.regul} gives
$\min_{K\in\T}\min_{\sigma\in\E_K}\dist(x_K,\sigma)/
\operatorname{diam}(K)\ge \zeta\dist_\sigma/\Delta x\ge \zeta\xi>0$.
It also implies the mesh regularity condition
$\operatorname{diam}(K)/\dist(x_K,\sigma)\le\xi_0$ in \cite[(9)]{EGH99},
since, because of \eqref{2.regul} again,
$\operatorname{diam}(K)/\dist(x_K,\sigma)\le \Delta x/(\zeta\dist_\sigma)
\le 1/(\xi\zeta)=:\xi_0$. It can be seen by considering quadratic cells that the quasi-uniform condition $\min_{\sigma\in\E_{\rm int}}\dist_\sigma\ge\xi\Delta x>0$ generally does not imply
the mesh regularity condition \eqref{2.regul} and vice versa, so
both conditions are independent from each other.
\end{remark}


\section{Proof of Theorem \ref{thm.large}}\label{sec.large} 

We infer from mass conservation, 
$\sum_{K\in\T}\m(K)u_{i,K}^k=\sum_{K\in\T}\m(K)u_{i,K}^0=\m(\Omega)\bar{u}_i$, that
\begin{align*}
  H(u^k-\bar{u},u^{k-1}-\bar{u})
	&= H(u^k,u^{k-1}) + \frac12\sum_{K\in\T}\m(K)\big(|\bar{u}|_A^2 - 3(u_K^k,\bar{u})_A + (u_K^{k-1},\bar{u})_A\big) \\
	&= H(u^k,u^{k-1}) - \frac12\m(\Omega)|\bar{u}|_A^2.
\end{align*}
Then the entropy inequality \eqref{2.ei.bdf} shows that
\begin{equation}\label{5.relH}
  H(u^k-\bar{u},u^{k-1}-\bar{u}) + \gamma\Delta t|A^{1/2}u^k|_{1,2,\T}^2
	\le H(u^{k-1}-\bar{u},u^{k-2}-\bar{u})
\end{equation}
for $k\ge 2$. Another iteration gives, for $k\ge 3$,
$$
  H(u^k-\bar{u},u^{k-1}-\bar{u}) + \gamma\Delta t
	\big(|A^{1/2}u^k|_{1,2,\T}^2 + |A^{1/2}u^{k-1}|_{1,2,\T}^2\big) \le H(u^{k-2}-\bar{u},u^{k-3}-\bar{u}).
$$
Hence, taking into account the discrete Poincar\'e--Wirtinger inequality \eqref{2.PWI},
\begin{align*}
  H(u^k-\bar{u},u^{k-1}-\bar{u}) 
    &+ \frac{\gamma\lambda_m\zeta}{C_P^2\lambda_M} 
    \Delta t\big(\|A^{1/2}(u^k-\bar{u})\|_{0,2,\T}^2 + \|A^{1/2}(u^{k-1}-\bar{u})\|_{0,2,\T}^2\big) \\
	&\le H(u^{k-2}-\bar{u},u^{k-3}-\bar{u}),
\end{align*}
and the norm equivalence \eqref{2.normineq},
$$
  H(u^k-\bar{u},u^{k-1}-\bar{u}) 
	+ \frac{4\gamma\lambda_m\zeta\Delta t}{(3+\sqrt{8})C_P^2\lambda_M}
	H(u^k-\bar{u},u^{k-1}-\bar{u}) \le H(u^{k-2}-\bar{u},u^{k-3}-\bar{u}).
$$
This can be written as
$$
  H(u^k-\bar{u},u^{k-1}-\bar{u}) \le (1+\kappa\Delta t)^{-1}H(u^{k-2}-\bar{u},u^{k-3}-\bar{u}),
$$
where $\kappa=4\gamma\lambda_m\zeta/((3+\sqrt{8})C_P^2\lambda_M)$.
Depending on whether $k$ is odd or even, we resolve this iteration as follows:
\begin{align*}
  H(u^{2\ell+1}-\bar{u},u^{2\ell}-\bar{u}) &\le (1+\kappa\Delta t)^{-\ell}
	H(u^1-\bar{u},u^{0}-\bar{u}), \\
	H(u^{2\ell+2}-\bar{u},u^{2\ell+1}-\bar{u}) &\le (1+\kappa\Delta t)^{-\ell}
	H(u^2-\bar{u},u^{1}-\bar{u}) \\ 
	&\le (1+\kappa\Delta t)^{-\ell}	H(u^1-\bar{u},u^{0}-\bar{u}),
\end{align*}
where we used \eqref{5.relH} in the last step. As in both cases
$\ell\ge (k-2)/2$, we conclude that
\begin{equation}\label{5.Hk}
  H(u^k-\bar{u},u^{k-1}-\bar{u}) \le (1+\kappa\Delta t)^{-(k-2)/2}H(u^1-\bar{u},u^{0}-\bar{u}).
\end{equation}
We want to express this inequality in terms of the $\|A^{1/2}(\cdot)\|_{0,2,\T}$ norm.
We observe that, by Young's inequality,
$\|A^{1/2}(u^k-\bar{u})\|_{0,2,\T}^2 \le 4H(u^k-\bar{u},u^{k-1}-\bar{u})$
and, in view of \eqref{2.ei.euler},
\begin{align*}
  H(u^1-\bar{u},u^0-\bar{u}) &= H(u^1)-H(\bar{u}) \le H(u^0)-H(\bar{u}) \\
  &= H(u^0-\bar{u}) = \frac12\|A^{1/2}(u^0-\bar{u})\|_{0,2,\T}^2.
\end{align*}
Then we deduce from \eqref{5.Hk} that
\begin{align*}
    \|A^{1/2}(u^k-\bar{u})\|_{0,2,\T}^2 
	&\le 4H(u^k-\bar{u},u^{k-1}-\bar{u}) 
	\le 4(1+\kappa\Delta t)^{-(k-2)/2}H(u^1-\bar{u},u^{0}-\bar{u}) \\
	&\le 2(1+\kappa\Delta t)^{-(k-2)/2}
    \|A^{1/2}(u^{0}-\bar{u})\|_{0,2,\T}^2,
\end{align*}
which concludes the proof.


\section{Proof of Theorem \ref{thm.conv}}\label{sec.conv}

We split the proof into two parts. We first prove the convergence in the space variable
and then the convergence in the time variable. An alternative is to show the convergence
in both variables simultaneously; see, e.g., \cite{JuZu22}.

\subsection{Convergence in space}

We show the following result for $\Delta x\to 0$.

\begin{proposition}[Convergence in space]\label{prop.space}
Let the assumptions of Theorem \ref{thm.conv} hold and let $(u_m^k)$ be the sequence
of solutions to \eqref{2.init}--\eqref{2.mean} constructed in Theorem \ref{thm.ex} associated to an admissible mesh $\T_m$ 
with mesh size $\Delta x_m$ for $m\in\N$
satisfying $\Delta x_m\to 0$ as $m\to\infty$.
Then there exists a subsequence which is not relabeled such that
$u_{i,m}^k\to u_i^k$ strongly in $L^2(\Omega)$ 
as $m\to\infty$ and $u_i^k$ solves for all
$\phi_i\in W^{1,\max\{2,d\}}(\Omega)$, $i=1,\ldots,n$,
\begin{equation}\label{5.BDF2}
  \frac{1}{\Delta t}\int_\Omega\bigg(\frac32 u_i^k - 2u_i^{k-1} + \frac12 u_i^{k-2}\bigg)
	\phi_i\dd x + \int_\Omega\big(\gamma\na u_i^k
	+ (u_i^k)^+\na p_i(u^k)\big)\cdot\na\phi_i\dd x = 0.
\end{equation}
\end{proposition}

\begin{proof}
For fixed $\Delta t$, the discrete entropy inequality in Theorem \ref{thm.ex} provides a uniform bound for $\|u_m^k\|_{1,2,\T_m}$. Then, by the discrete Rellich--Kondrachov
compactness theorem \cite[Lemma 5.6]{EGH00},
there exists a subsequence of $(u_m^k)=(u_{1,m}^k,\ldots,u_{n,m}^k)$, 
which is not relabeled, such that
$u_m^k\to u^k$ strongly in $L^2(\Omega)$ as $m\to\infty$. Moreover, the sequence
of discrete gradients $(\na^m u_m^k)$ converges weakly in $L^2(\Omega)$ to some
function which can be identified by $\na u^k$; see \cite[Lemma 4.4]{CLP03}.
Let $\phi_i\in C^2(\overline\Omega)$ and set $\phi_{i,K}:=\phi_i(x_K)$ for $K\in\T$. Then the limit $\Delta x_m\to 0$ in the BDF2 approximation becomes
$$
  \frac{1}{\Delta t}\sum_{K\in\T}\m(K)\bigg(\frac32 u_{i,K}^k - 2u_{i,K}^{k-1}
	+ \frac12 u_{i,K}^{k-2}\bigg)\phi_{i,K} \to
	\frac{1}{\Delta t}\int_\Omega\bigg(\frac32 u_i^k - 2u_i^{k-1} + \frac12 u_i^k\bigg)\phi_i\dd x.
$$
Next, we set $F^m=F_1^m+F_2^m+F_3^m$, where
\begin{align*}
  F_1^m &= -\gamma\sum_{K\in\T}\sum_{\sigma\in\E_K}\tau_\sigma
  \mathrm{D}_{K,\sigma}u_{i,m}^k\phi_{i,K}, \\
  F_2^m &= -\sum_{K\in\T}\sum_{\sigma\in\E_K}\tau_\sigma(u_{i,m,K}^k)^+
	\mathrm{D}_{K,\sigma}p_i(u_m^k)\phi_{i,K}, \\
	F_3^m &= -\sum_{K\in\T}\sum_{\sigma\in\E_K}\tau_\sigma\big(
	(u_{i,m,\sigma}^k)^+ - (u_{i,m,K}^k)^+\big)\mathrm{D}_{K,\sigma}p_i(u_m^k)\phi_{i,K}.
\end{align*}
We introduce the intermediate integral $F_0^m=F_{01}^m+F_{02}^m$, where
$$
  F_{01}^m = \gamma\int_\Omega\na^m u_{m,i}^k\cdot\na\phi_i\dd x, \quad
  F_{02}^m = \int_\Omega(u_{i,m}^k)^+\na^m p_i(u_m^k)\cdot\na\phi_i\dd x.
$$
It follows from the weak convergence of the discrete gradients and the strong convergence in $L^2(\Omega)$ that $F_0^m\to F$ as $m\to\infty$, where
$$
  F = \gamma\int_\Omega\na u_i^k\cdot\na\phi_i\dd x
	+ \int_\Omega(u_i^k)^+\na p_i(u^k)\cdot\na\phi_i\dd x.
$$
Thus, if we can show that $F_0^m-F^m\to 0$, then
$|F^m-F|\le|F^m-F_0^m|+|F_0^m-F|\to 0$, proving the claim.

By discrete integration by parts and the
definition of the discrete gradient,
\begin{align*}
  F_{1}^m &= \gamma\sum_{\substack{\sigma\in\E_{\rm int} \\ 
  \sigma=K|L}}\tau_\sigma
  \mathrm{D}_{K,\sigma}u_{i,m}^k\mathrm{D}_{K,\sigma}\phi_i, \\
  F_{01}^m &= \gamma\sum_{\substack{\sigma\in\E_{\rm int} \\ \sigma=K|L}}
	\frac{\m(\sigma)}{\m(T_{K,\sigma})}\mathrm{D}_{K,\sigma}u_{i,m}^k\int_{T_{K,\sigma}}
	\na\phi_i\cdot\nu_{K,\sigma}\dd x.
\end{align*}
Using the Taylor expansion (here we need $\phi_i\in C^2(\overline\Omega)$)
$$
  \frac{\mathrm{D}_{K,\sigma}\phi_i}{\dist_\sigma}
	= \frac{\phi_{i,L}-\phi_{i,K}}{\dist(x_K,x_L)} = \na\phi_i\cdot\nu_{K,\sigma}
	+ \mathcal{O}(\Delta x_m)\quad\mbox{for }\sigma=K|L,
$$
where we have taken into account the property $x_K-x_L=\dist(x_K,x_L)\nu_{K,\sigma}$,
we obtain 
\begin{align*}
  |F_{01}^m-F_1^m| &\le \gamma\sum_{\substack{\sigma\in\E_{\rm int} \\ \sigma=K|L}}\m(\sigma)
	|\mathrm{D}_{K,\sigma}u_{i,m}^k|\bigg|\frac{1}{\m(T_{K,\sigma})}\int_{T_{K,\sigma}}
	\na\phi_i\cdot\nu_{K,\sigma}\dd x - \frac{\mathrm{D}_{K,\sigma}\phi_i}{\dist_\sigma}\bigg| \\
  &\le C\gamma\Delta x_m\sum_{\sigma\in\E_{\rm int}}
  \m(\sigma)|\mathrm{D}_\sigma u_{i,m}^k|,
\end{align*}
where $C>0$ depends on the $L^\infty$ norm of $\mathrm{D}^2\phi_i$. We apply the Cauchy--Schwarz inequality and use the mesh property \eqref{2.sum} to find that
\begin{align*}
  |F_{01}^m-F_1^m| &\le C\gamma\Delta x_m\bigg(\sum_{\sigma\in\E_{\rm int}}
    \frac{\m(\sigma)}{\dist_\sigma}|\mathrm{D}_\sigma u_{i,m}^k|^2
    \bigg)^{1/2}
	\bigg(\sum_{\sigma\in\E_{\rm int}}\m(\sigma)\dist_\sigma\bigg)^{1/2} \\
	&\le C\gamma\Delta x_m|u_{i,m}^k|_{1,2,\T_m}\bigg(\frac{d}{\zeta}\m(\Omega)\bigg)^{1/2}
	\to 0\quad\mbox{as }m\to\infty.
\end{align*}
Similar arguments lead to
\begin{align*}
  |F_{02}^m-F_2^m| &\le C\Delta x_m\sum_{K\in\T_m}\sum_{\sigma\in\E_{{\rm int},K}}\m(\sigma)
	(u_{i,m,K}^k)^+|\mathrm{D}_{K,\sigma}p_i(u_m^k)| \\
	&\le C\Delta x_m\bigg(\sum_{K\in\T_m}|(u_{i,m,K}^k)^+|^2\sum_{\sigma\in\E_{{\rm int},K}}
	\m(\sigma)\dist_\sigma\bigg)^{1/2}|p_i(u_m^k)|_{1,2,\T_m} \\
	&\le C\Delta x_m\bigg(\frac{d}{\zeta}\sum_{K\in\T_m}\m(K)|(u_{i,m,K}^k)^+|^2\bigg)^{1/2}
	|p_i(u_m^k)|_{1,2,\T_m} \\
	&\le C(\zeta)\Delta x_m\|u_{i,m}^k\|_{0,2,\T_m}|p_i(u_m^k)|_{1,2,\T_m}.
\end{align*}
The right-hand side converges to zero since 
$$
  |p_i(u_m^k)|_{1,2,\T_m}^2 = \sum_{\substack{\sigma\in\E_{\rm int} \\ \sigma=K|L}}
	\tau_\sigma \bigg(\sum_{j=1}^n a_{ij}\mathrm{D}_{K,\sigma}u_{j,m}^k\bigg)^2
	\le C(A)|u_{m}^k|_{1,2,\T_m}^2 \le C.
$$
Finally, using $|\mathrm{D}_{K,\sigma}\phi_i|\le C(\phi_i)\Delta x_m$
and property (ii) of the mean function, 
\begin{align*}
  |F_3^m| &\le \sum_{\substack{\sigma\in\E_{\rm int} \\ \sigma=K|L}}
  \tau_\sigma|u_{i,m,\sigma}^k - u_{i,m,K}^k||\mathrm{D}_{K,\sigma}p_i(u_m^k)|
  |\mathrm{D}_{K,\sigma}\phi_i| \\
  &\le C(\phi_i)\Delta x_m
  \sum_{\substack{\sigma\in\E_{\rm int} \\ \sigma=K|L}}\tau_\sigma
  |\mathrm{D}_\sigma u_{i,m}^k||\mathrm{D}_{K,\sigma}p_i(u_m^k)| \\
  &\le C(\phi_i,A)\Delta x_m\bigg(
  \sum_{\sigma\in\E}\tau_\sigma
  |\mathrm{D}_\sigma u_{i,m}^k|^2\bigg)^{1/2}
  \bigg(\sum_{j=1}^n\sum_{\sigma\in\E}
  \tau_\sigma|\mathrm{D}_\sigma u_{j,m}^k|^2\bigg)^{1/2}\to 0.
\end{align*}
This shows that $F_0^m-F\to 0$ as $m\to\infty$, concluding the proof.
\end{proof}

\subsection{Convergence in time}

We wish to perform the limit $\Delta t\to 0$ in \eqref{5.BDF2}.
For this, we need an estimate in a better space than $L^2(\Omega_T)$, provided by the
following lemma.

\begin{lemma}[Higher-order integrability]\label{lem.hoi}
Let $(u^{(\tau)})$ be a family of solutions to \eqref{5.BDF2} associated to
the time step size $\tau:=\Delta t$,
constructed in Proposition \ref{prop.space}. Then there exists $C>0$ independent of $\tau$
such that 
$$
  \|u^{(\tau)}\|_{L^p(\Omega_T)}\le C\quad\mbox{for }p=2+4/d.
$$
\end{lemma}
 
\begin{proof}
The lemma follows from the discrete entropy inequalities 
\eqref{2.ei.bdf}--\eqref{2.ei.euler} and the 
Gagliar\-do--Nirenberg inequality. Indeed, we infer from the
entropy inequalities after summation over $k=2,\ldots,N_T$ that
$$
  \|u^{(\tau)}\|_{L^\infty(0,T;L^2(\Omega))} 
	+ \|u^{(\tau)}\|_{L^2(0,T;H^1(\Omega))} \le C.
$$ 
Then it follows from the Gagliardo--Nirenberg inequality with $\theta=d/2-d/p$ that
\begin{align*}
  \|u^{(\tau)}\|_{L^p(0,T;L^p(\Omega))}^p 
	&\le C\int_0^T\|u^{(\tau)}\|_{H^1(\Omega)}^{p\theta}
	\|u^{(\tau)}\|_{L^2(\Omega)}^{p(1-\theta)}\dd t \\
	&\le C\|u^{(\tau)}\|_{L^\infty(0,T;L^2(\Omega))}^{p(1-\theta)}
	\int_0^T\|u^{(\tau)}\|_{H^1(\Omega)}^2\dd t\le C,
\end{align*}
since $p\theta=2$. This finishes the proof.
\end{proof}

\begin{proposition}[Convergence in time]\label{prop.time}
Let $(u^{(\tau)})$ be a family of solutions to \eqref{5.BDF2} with $\tau=\Delta t$. Then
$u^{(\tau)}$ converges to a weak solution $u$ to \eqref{1.eq}--\eqref{1.bic}
satisfying 
$$
  u_i\in L^2(0,T;H^1(\Omega))\cap L^\infty(0,T;L^2(\Omega)), \quad
  \pa_t u_i\in L^{r}(0,T;W^{1,2d+4}(\Omega)'),
$$
where $r=(2d+4)/(2d+3)>1$.
\end{proposition}

\begin{proof}
We estimate the discrete time derivative
$\mathrm{D}_{\tau} u_i^{(\tau)}(t):= \frac32 u_i^k-2u_i^{k-1}+\frac12 u_i^{k-2}$
for $t\in[k\tau,(k+1)\tau)$ for $k\ge 2$. Let $\phi_i\in L^{2d+4}(0,T;W^{1,2d+4}(\Omega))$. Then
\begin{align*}
  \frac{1}{\tau}&\int_{2\tau}^T\big|\langle\mathrm{D}_\tau u_i^{(\tau)},\phi_i\rangle_{W^{1,d+2}(\Omega)'}\big|^r\dd t \\
	&\le \gamma^r C\int_{2\tau}^T\int_\Omega
    |\na u_i^{(\tau)}\cdot\na\phi_i|^r\dd x\dd t 
	+ C\int_{2\tau}^T\int_\Omega|(u_i^{(\tau)})^+\na p_i(u^{(\tau)})\cdot\na\phi_i|^r \dd x\dd t \\
	&\le \gamma^r C\|\na u_i^{(\tau)}\|_{L^2(\Omega_T)}^r
    \|\na\phi_i\|_{L^{2d+4}(\Omega_T)}^r \\
	&\phantom{xx}{}
	+ C\|u_i^{(\tau)}\|_{L^{(2d+4)/d}(\Omega_T)}^r
    \|\na p_i(u^{(\tau)})\|_{L^{2}(\Omega_T)}^r
	\|\na\phi_i\|_{L^{2d+4}(\Omega_T)}^r \\
	&\le C\|\phi_i\|_{L^{2d+4}(0,T;W^{1,2d+4}(\Omega))}^r,
\end{align*}
where we used the fact that $p_i(u^{(\tau)})$ is a linear combination of all
$u_j^{(\tau)}$ for $j=1,\ldots,n$. This implies the bound
$\tau^{-1}\|\mathrm{D}_\tau u_i^{(\tau)}\|_{L^r(2\tau,T;W^{1,2d+4}(\Omega)')}\le C$.

Let $\pi_\tau u^{(\tau)}(t)=u^{(\tau)}(t-\tau)$ be a shift operator.
We relate the implicit Euler scheme and the BDF2 scheme by
$$
  u^k_i-u_i^{k-1} = \frac23\bigg(\frac32 u_i^k-2u_i^{k-1}+\frac12u_i^{k-2}\bigg)
	+ \frac13(u_i^{k-1}-u_i^{k-2}).
$$
Then
\begin{align*}
  \|u^{(\tau)}&-\pi_\tau u^{(\tau)}\|_{L^{r}(2\tau,T;
  W^{1,2d+4}(\Omega)')}
	= \Big\|\frac23\mathrm{D}_\tau u^{(\tau)} + \frac13\pi_\tau(u^{(\tau)}-\pi_\tau u^{(\tau)})
	\Big\|_{L^{r}(2\tau,T;W^{1,2d+4}(\Omega)')} \\
	&\le \frac23\|\mathrm{D}_\tau u^{(\tau)}\|_{L^{r}(2\tau,T;W^{1,2d+4}(\Omega)')}
	+ \frac13\|u^{(\tau)}-\pi_\tau u^{(\tau)}\|_{L^{r}(\tau,T-\tau;W^{1,2d+4}(\Omega)')}.
\end{align*}
Adding $\|u^{(\tau)}-\pi_\tau u^{(\tau)}\|_{L^{r}(2\tau,T;W^{1,2d+4}(\Omega)')}\le C_1$ 
from the first Euler step (proved in a similar way as above) 
to the left-hand side
and absorbing the last term on the right-hand side by the left-hand side, we find that
$$
  \frac{2}{3\tau}\|u^{(\tau)}-\pi_\tau u^{(\tau)}\|_{L^{r}(2\tau,T;W^{1,2d+4}(\Omega)')}
	\le \frac{2}{3\tau}\|\mathrm{D}_\tau u^{(\tau)}\|_{L^{r}(2\tau,T;W^{1,2d+4}(\Omega)')}
	\le C.
$$
Together with the uniform $L^2(0,T;H^1(\Omega))$ bound for $u^{(\tau)}$, we can apply
the Aubin--Lions compactness lemma in the version of \cite{DrJu12} to conclude that,
up to a subsequence, as $\tau\to 0$,
$$
  u^{(\tau)}\to u\quad\mbox{strongly in }L^2(\Omega_T).
$$
In view of the higher-order estimate of Lemma \ref{lem.hoi}, this convergence also holds
in $L^q(\Omega_T)$ for all $q<2+4/d$. Furthermore, again up to a subsequence, 
$$
  \mathrm{D}_\tau u^{(\tau)}\rightharpoonup \pa_t u\quad\mbox{weakly in }
	L^{r}(2\tau,T;W^{1,2d+4}(\Omega)').
$$
These convergences are sufficient to pass to the limit $\tau\to 0$ in \eqref{5.BDF2} for test functions $\phi_i\in L^{2d+4}(2\tau,T;
W^{1,2d+4}(\Omega)')$.
\end{proof}


\section{Second-order convergence}\label{sec.second}

As in the previous section, we set
$\mathrm{D}_{\Delta t}u_i^k = \frac32 u_i^k - 2u_i^{k-1} + \frac12u_i^{k-2}$
and write \eqref{5.BDF2} as
\begin{equation}\label{6.uk}
  \frac{1}{\Delta t}\int_\Omega\mathrm{D}_{\Delta t}u_i^k\phi_i\dd x
	+ \int_\Omega\big(\gamma\na u_i^k + (u_i^k)^+\na p_i(u^k)\big)
    \cdot\na\phi_i\dd x	= 0.
\end{equation}
A Taylor expansion shows that, for some $\xi_k\in(0,T)$,
$$
  \mathrm{D}_{\Delta t}u_i(t_k):=\frac32 u_i(t_k)-2u_i(t_{k-1})+\frac12 u_i(t_{k-2})
	= (\Delta t)\pa_t u_i(t_k) - \frac{(\Delta t)^3}{3}\pa_t^3 u_i(\xi_k).
$$
Then, using a test function $\phi_i\in H^1(\Omega)$ in \eqref{1.eq},
\begin{equation}\label{6.u}
  \frac{1}{\Delta t}\int_\Omega\mathrm{D}_{\Delta t}u_i(t_k)\phi_i\dd x
	+ \int_\Omega(\gamma\na u_i + u_i\na p_i(u))(t_k)\cdot\na\phi_i\dd x
	= \frac{(\Delta t)^2}{3}\int_\Omega\pa_t^3 u_i(\xi_k)\phi_i\dd x.
\end{equation}
We take the difference of \eqref{6.uk} and \eqref{6.u}, 
choose the test function $\phi_i=p_i(u(t_k))-p_i(u^k)=(A(u(t_k)-u^k))_i$,
and sum over $i=1,\ldots,n$:
\begin{align}\label{6.aux}
  & \frac{1}{\Delta t}\int_\Omega\mathrm{D}_{\Delta t}(u(t_k)-u^k)^TA(u(t_k)-u^k)\dd x = I_7 + I_8, \quad\mbox{where}, \\
  & I_7 = -\sum_{i=1}^n\int_\Omega\big[\gamma\na(u_i(t_k)-u_i^k) 
  + u_i(t_k)\na p_i(u(t_k))-(u_i^k)^+\na p_i(u^k)\big] \nonumber \\
  &\phantom{xxxx}{\times}\na(A(u(t_k)-u^k))_i\dd x, \nonumber \\
  & I_8 = \frac{(\Delta t)^2}{3}\sum_{i=1}^n\int_\Omega\pa_t^3 u_i(\xi_k)(A(u(t_k)-u^k))_i\dd x. \nonumber
\end{align}
Set $v^k:=u(t_k)-u_i^k$. It follows from the BDF2 inequality in Lemma \ref{lem.bdf2}, applied to the left-hand side, 
that
$$
  \frac{1}{\Delta t}\int_\Omega\mathrm{D}_{\Delta t}(u(t_k)-u^k)^TA(u(t_k)-u^k)\dd x
	\ge \frac{1}{\Delta t}\big(H(v^k,v^{k-1})-H(v^{k-1},v^{k-2})\big).
$$
For the terms $I_7$ and $I_8$, we use the definition $p_i(u^k)=(Au^k)_i$, the Lipschitz continuity of $z\mapsto z^+$, the nonnegativity of $u_i$, and Young's inequality:
\begin{align*}
  I_7 &= -\sum_{i=1}^n\int_\Omega\gamma\na(A^{1/2}v^k)_i\cdot\na(A^{1/2}v^k)_i\dd x \\
	&\phantom{xx}{}- \sum_{i=1}^n\int_\Omega\big((u_i(t_k)-(u_i^k)^+)\na(Au(t_k))_i
	+ (u_i^k)^+\na(A(u(t_k)-u^k))_i\big)\cdot\na(Av^k)_i\dd x \\
	&\le -\gamma\|\na(A^{1/2}v^k)\|_{L^2(\Omega)}^2
	+ \lambda_m^{-1/2}\|A^{1/2}v^k\|_{L^2(\Omega)}\lambda_M^{3/2}\|\na u(t_k)\|_{L^\infty(\Omega)}
	\|\na(A^{1/2}v^k)\|_{L^2(\Omega)} \\
    &\phantom{xx}{}- \sum_{i=1}^n\int_\Omega(u_i^k)^+|\na(Av^k)_i|^2\dd x
	\le \frac{\lambda_M^3}{4\gamma\lambda_m}\|\na u\|_{L^\infty(\Omega_T)}^2
	\|A^{1/2}v^k\|_{L^2(\Omega)}^2 \quad\mbox{and} \\
	I_8 &\le \frac{(\Delta t)^2}{3\lambda_m^{1/2}}\|\pa_t^3 u\|_{L^\infty(0,T;L^2(\Omega))}\|A^{1/2}v^k\|_{L^2(\Omega)}.
\end{align*}
Summarizing, we obtain from \eqref{6.aux}
\begin{align}\label{6.aux3}
  & H(v^k,v^{k-1})-H(v^{k-1},v^{k-2}) \le C_1\Delta t	\|A^{1/2}v^k\|_{L^2(\Omega)}^2 
	+ C_2(\Delta t)^3\|A^{1/2}v^k\|_{L^2(\Omega)}, \\
  &\mbox{where}\quad
  C_1 = \frac{\lambda_M^3}{4\gamma\lambda_m}\|\na u\|_{L^\infty(\Omega_T)}^2, \quad
	C_2 = \frac{1}{3\lambda_m^{1/2}}\|\pa_t^3 u\|_{L^\infty(0,T;L^2(\Omega))}. \nonumber
\end{align}
We iterate this inequality once more and use the inequality $a+b\le\sqrt{2(a^2+b^2)}$
as well as the norm equivalence \eqref{2.hineq}:
\begin{align*}
  H(v^k,v^{k-1})-H(v^{k-2},v^{k-3}) &\le C_1\Delta t\big(\|A^{1/2}v^k\|_{L^2(\Omega)}^2 
	+ \|A^{1/2}v^{k-1}\|_{L^2(\Omega)}^2\big) \\
  &\phantom{xx}{}+ C_2(\Delta t)^3\big(\|A^{1/2}v^k\|_{L^2(\Omega)} 
	+ \|A^{1/2}v^{k-1}\|_{L^2(\Omega)}\big) \\
	&\le C_1\Delta t\big(\|A^{1/2}v^k\|_{L^2(\Omega)}^2 
	+ \|A^{1/2}v^{k-1}\|_{L^2(\Omega)}^2\big) \\
  &\phantom{xx}{}+ \sqrt{2}C_2(\Delta t)^3\big(\|A^{1/2}v^k\|_{L^2(\Omega)}^2
	+ \|A^{1/2}v^{k-1}\|_{L^2(\Omega)}^2\big)^{1/2} \\
	&\le \frac{4C_1\Delta t}{3-\sqrt{8}}H(v^k,v^{k-1}) 
	+ \frac{4\sqrt{2}C_2(\Delta t)^3}{3-\sqrt{8}} H(v^k,v^{k-1})^{1/2}.
\end{align*}
We apply Young's inequality for $\eps>0$:
$$
  \bigg(1 - \frac{4(C_1+\eps)}{3-\sqrt{8}}\Delta t\bigg)H(v^k,v^{k-1})
	\le H(v^{k-2},v^{k-3})
    + \frac{2C_2^2(\Delta t)^5}{(3-\sqrt{8})\eps},
$$
and assume that $\Delta t < (3-\sqrt{8})/(4(C_1+\eps)$).
This recursion is of the form $a_k\le ba_{k-2}+bc(\Delta t)^5$, where
$a_k=H(v^k,v^{k-1})$ and
$$
  b = \bigg(1 -  \frac{4(C_1+\eps)}{3-\sqrt{8}}\Delta t\bigg)^{-1}, \quad
	c = \frac{2C_2^2(\Delta t)^5}{(3-\sqrt{8})\eps},
$$
and it can be resolved explicitly depending on whether $k$ is odd or even:
$$
  a_{2\ell+1}\le b^\ell a_1 + c(\Delta t)^5\sum_{j=0}^{\ell-1}b^j, \quad
	a_{2\ell+2}\le b^\ell a_2 + c(\Delta t)^5\sum_{j=0}^{\ell-1}b^j.
$$
The sum can be estimated according to
$$
  \sum_{j=0}^{\ell-1}b^j = \frac{b^\ell-1}{b-1} 
	\le \bigg(1 - \frac{4\Delta t}{3-\sqrt{8}}(C_1+\eps)\bigg)^{-\ell+1}
	 \frac{3-\sqrt{8}}{4\Delta t(C_1+\eps)}.
$$
Since $\ell=t_\ell/\Delta t\le T/\Delta t$, the bracket approximates the exponential function and can be bounded by a constant depending only 
on $C_1+\eps$ and $T$.
This shows that there exist constants $K_1$, $K_2>0$ such that
\begin{align*}
  H(v^{2\ell+1},v^{2\ell}) &\le K_1(C_1,\eps,T)H(v^1,v^0) 
	+ K_2(C_1,C_2,\eps^{-1},T)(\Delta t)^4, \\
  H(v^{2\ell+2},v^{2\ell+1}) &\le K_1(C_1,\eps,T)H(v^2,v^1) 
	+ K_2(C_1,C_2,\eps^{-1},T)(\Delta t)^4.
\end{align*}
Going back to inequality \eqref{6.aux3} for $k=2$, we can argue in a similar way as before
that $H(v^2,v^1)$ is bounded by $K_3 H(v^1,v^0)+K_4(\Delta t)^5$ 
for some constants $K_3$, $K_4>0$, which are independent of $\Delta t$.
Furthermore, since $v^0=0$, we have $H(v^1,v^0) = (5/4)\|A^{1/2}(u(t_1)-u^1)\|_{L^2(\Omega)}
\le K_5(\Delta t)^4$ for some $K_5>0$ independent of $\Delta t$. This shows that $H(v^k,v^{k-1})\le K_6(\Delta t)^4$, where $K_6$ depends on
$C_1$, $C_2$, $\eps^{-1}$, and $T$.
Taking the square root and using \eqref{2.hineq}
shows the result.


\section{Numerical examples}\label{sec.num}

The finite-volume scheme \eqref{2.init}--\eqref{2.mean} is implemented in Matlab, using the mobility $M(u,v)$ $=\frac12(u+v)$.
As the numerical scheme is implicit, we have solved the nonlinear system of equations at each time step by using the Matlab routine
{\tt fsolve}, based on Newton's method with trust regions. The optimality tolerance was chosen as $10^{-14}$.

\subsection{First example: one-dimensional domain, three species}

We choose the domain $\Omega=(0,1)$, the parameter $\gamma=1/2$, as well as the positive definite matrix $A$ and the initial 
data $u^0$ according to
$$
  A = \begin{pmatrix}
	2 & 1 & 1/2 \\ 1 & 3 & 3/2 \\ 1/2 & 3/2 & 1 
	\end{pmatrix}, \quad 
  u^0(x) = \begin{pmatrix} \cos(\pi x)+2 \\ 2-\cos(2\pi x) \\ 2 \end{pmatrix}.
$$
The numerical parameters are $\Delta x=1/12\,800$ and 
$\Delta t=1/128$. 
The numerical solution is illustrated in Figure \ref{fig.one} at various times.
All components converge to the constant steady state $\bar{u}=2$.
Interestingly, although initially equal to the steady state, the density $u_3$ 
becomes nonconstant for positive times before it tends to the constant steady
state for large times. Such a phenomenon is sometimes called uphill diffusion,
which typically appears in thermodynamic multicomponent systems due to cross diffusion \cite{Kri15}.

\begin{figure}[ht]
\hspace*{-6mm}
\includegraphics[width=0.35\textwidth]{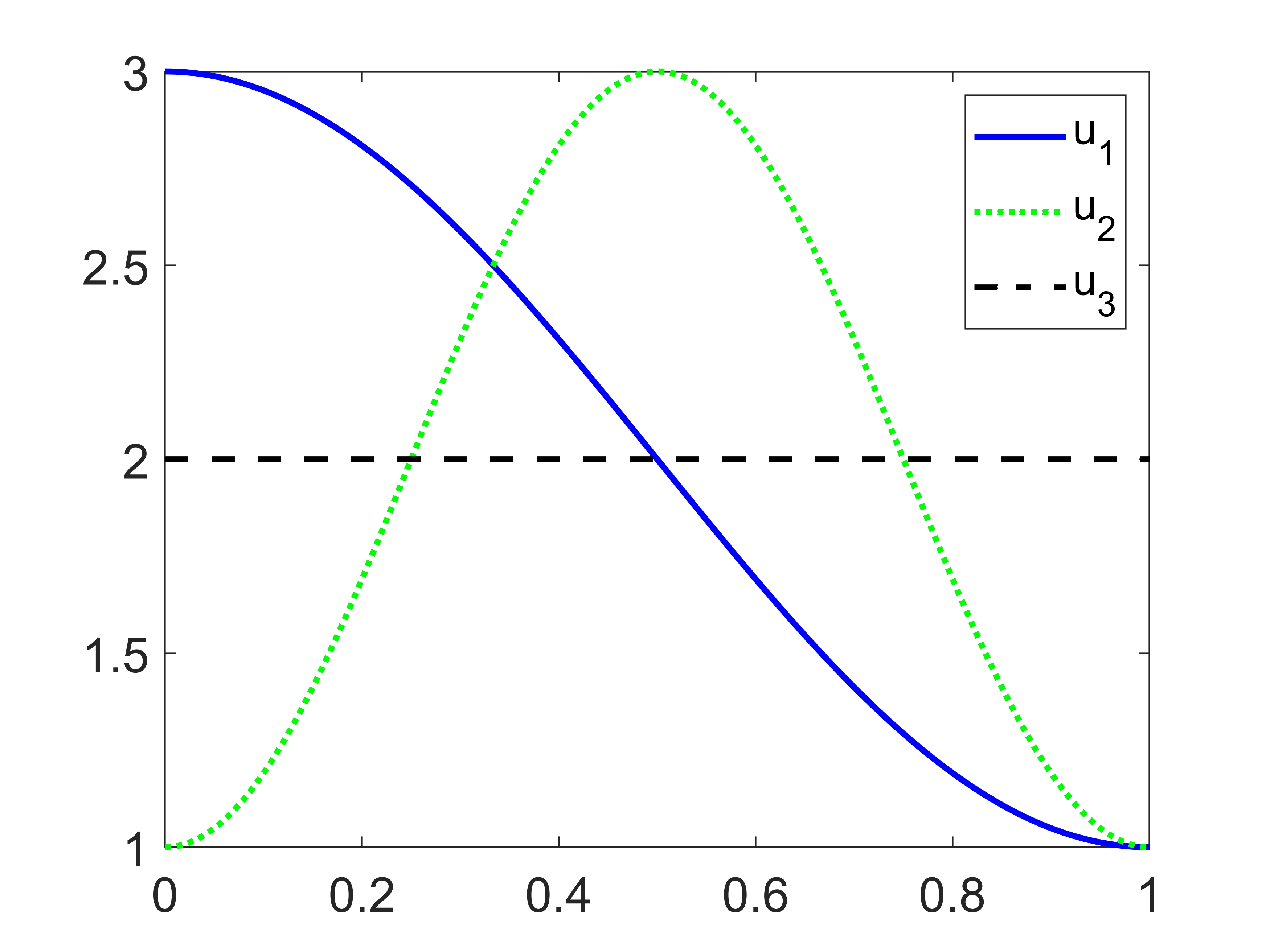}
\hskip-5mm
\includegraphics[width=0.35\textwidth]{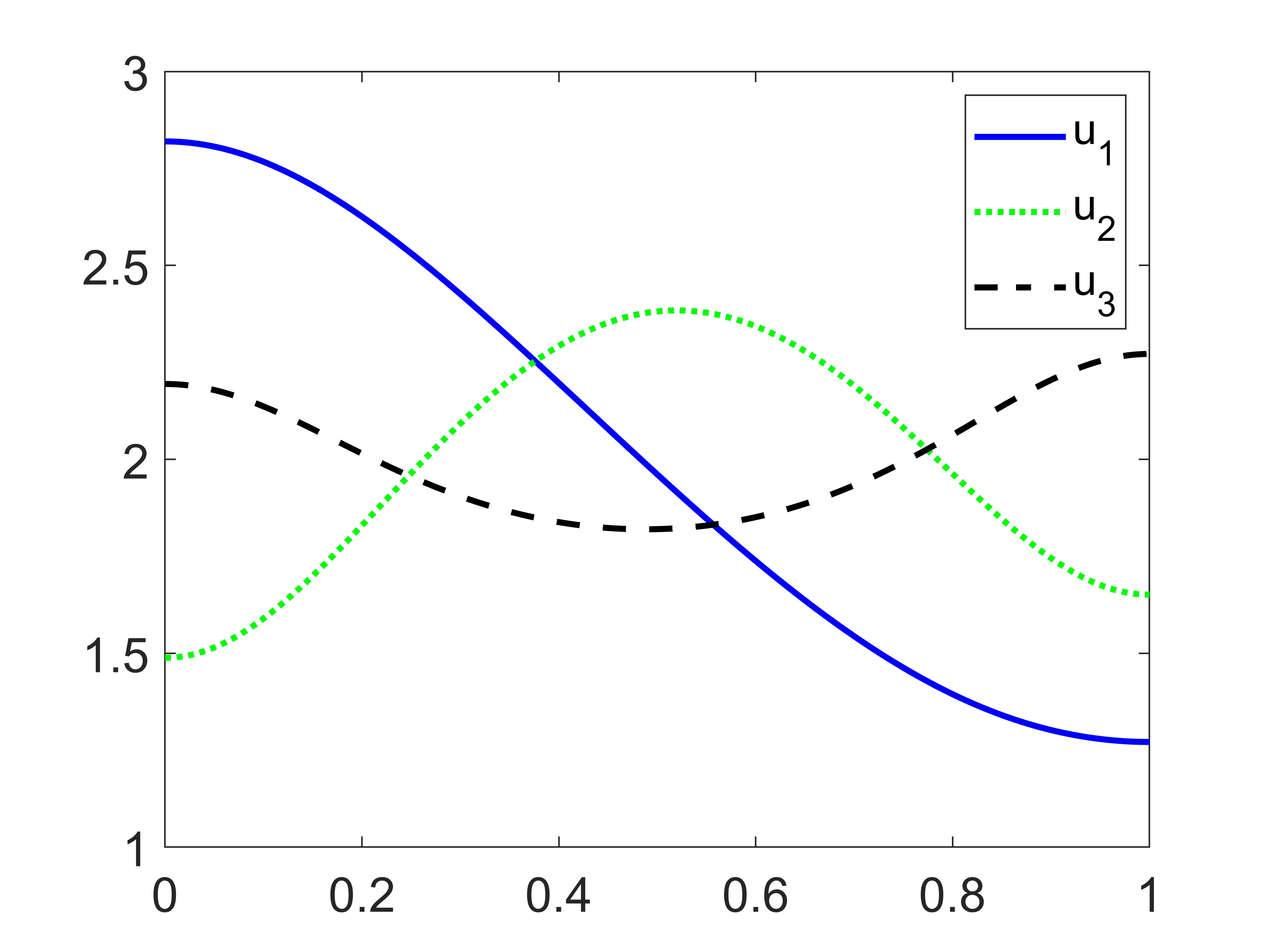}
\hskip-5mm
\includegraphics[width=0.35\textwidth]{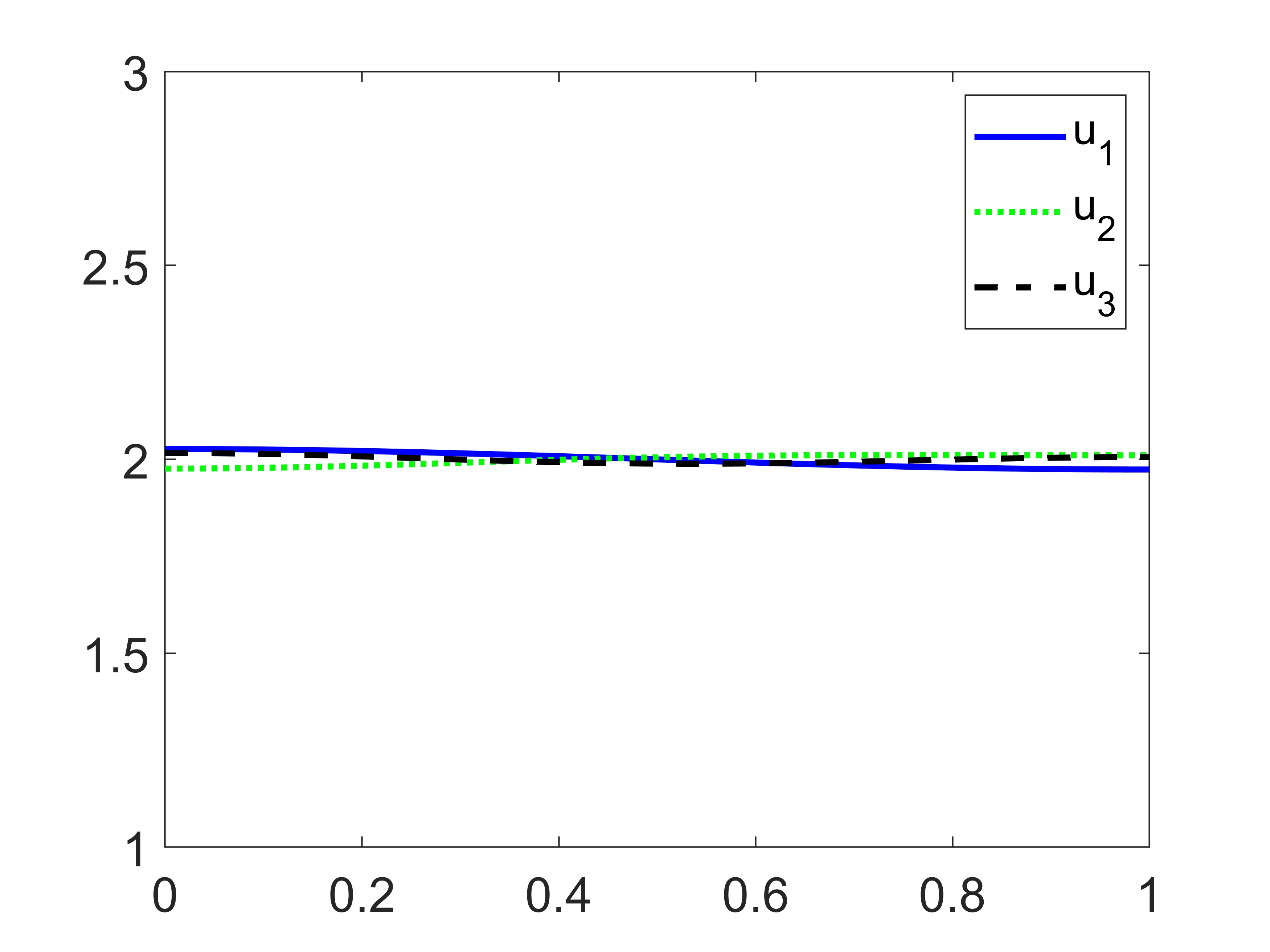}
\caption{Densities $u_1(t)$ (darker blue line), $u_2(t)$ (lighter green line), $u_3(t)$ (dashed black line) at times
$t=0,0.01,0.1$ (from left to right) versus space.}
\label{fig.one}
\end{figure}

\subsection{Second example: two-dimensional domain, two species}

We take $\Omega=(0,1)^2$, $\Delta x=\sqrt{2}\cdot 2^{-5}\approx 0.0044$, $\Delta t=1/256$, $\gamma=1/2$, and
$$
  A = \begin{pmatrix} 1 & 1/2 \\ 1/2 & 1 \end{pmatrix}, \quad
  u^0(x) = \begin{pmatrix} \mathrm{1}_{(0,1/2)^2}(x) \\ \mathrm{1}_{(1/2,1)^2}(x)
  \end{pmatrix}.
$$
Figure \ref{fig.two} shows the evolution of $u=(u_1,u_2)$ at various times.
Although being discontinuous and segregated initially, the solution becomes 
smooth and mixes the densities for positive times. This is not surprising, as
full segregation (i.e., the supports of $u_1$ and $u_2$ do not intersect)
is expected only when $\gamma=0$ and $\det A=0$. The numerical scheme preserved the nonnegativity in all our experiments, even for the initial data of this example. The numerical solutions are the same with or without the cutoff used in \eqref{2.flux}.

\begin{figure}[ht]
\includegraphics[width=0.32\textwidth]{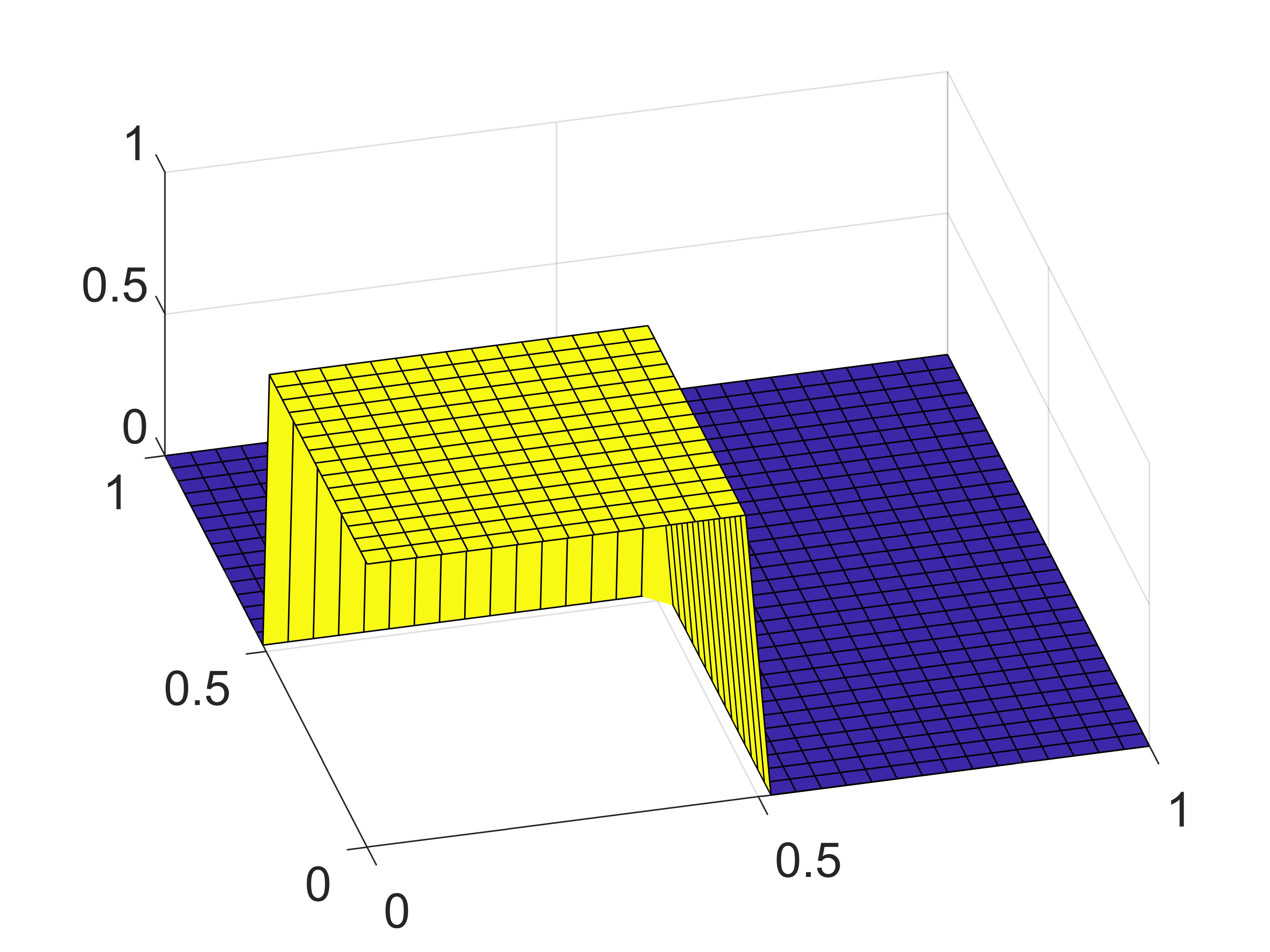}
\includegraphics[width=0.32\textwidth]{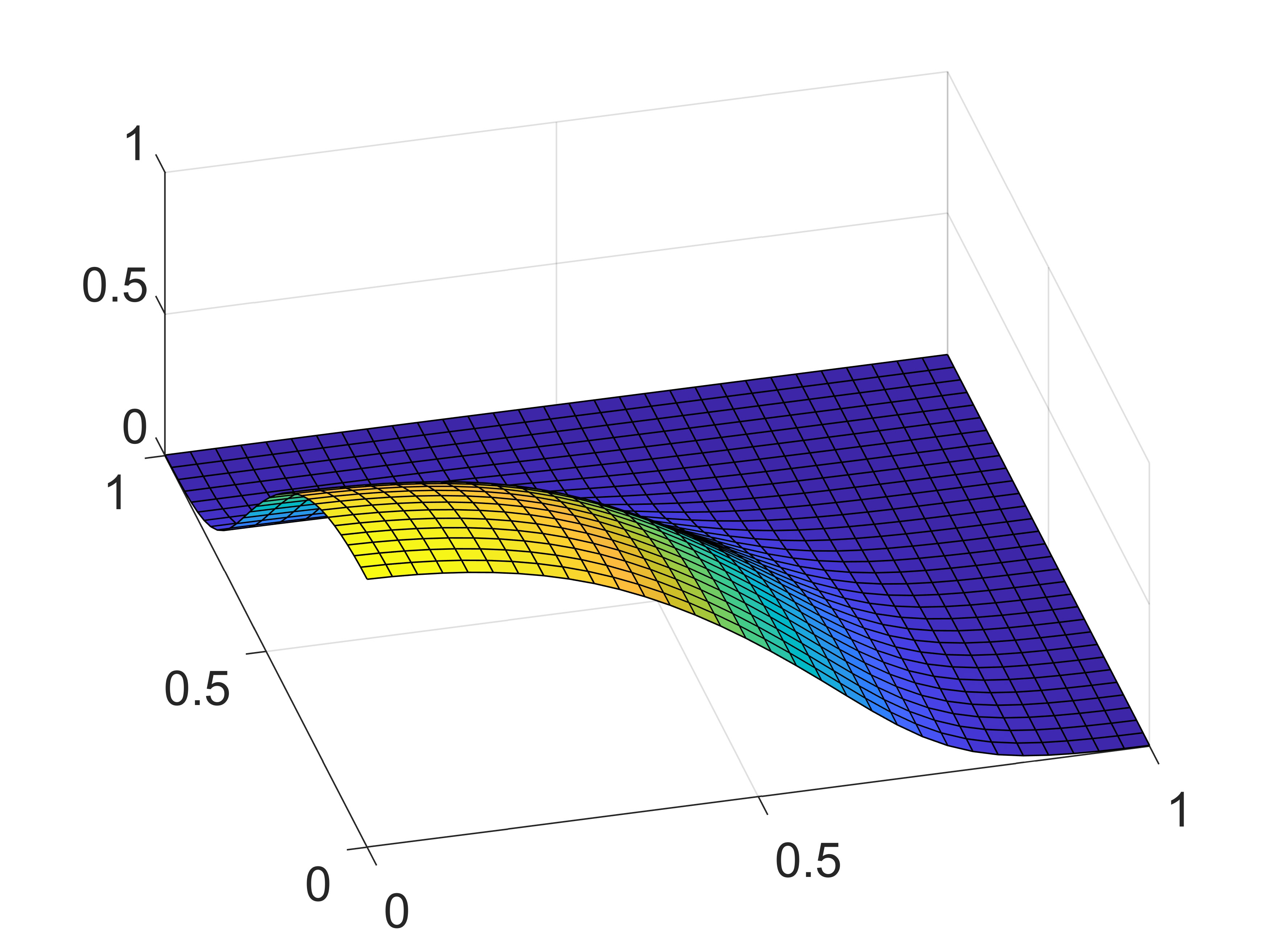}
\includegraphics[width=0.32\textwidth]{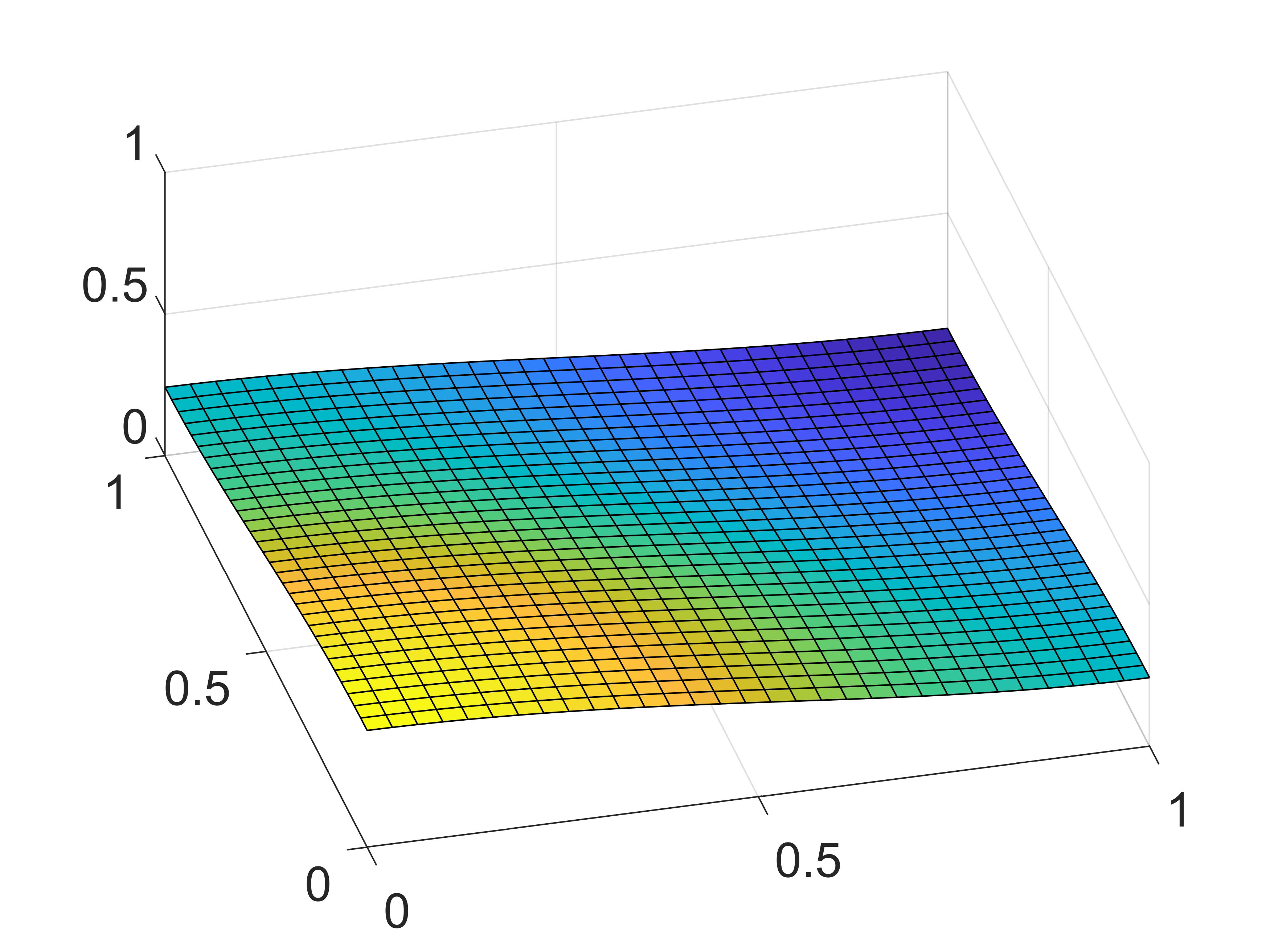}
\includegraphics[width=0.32\textwidth]{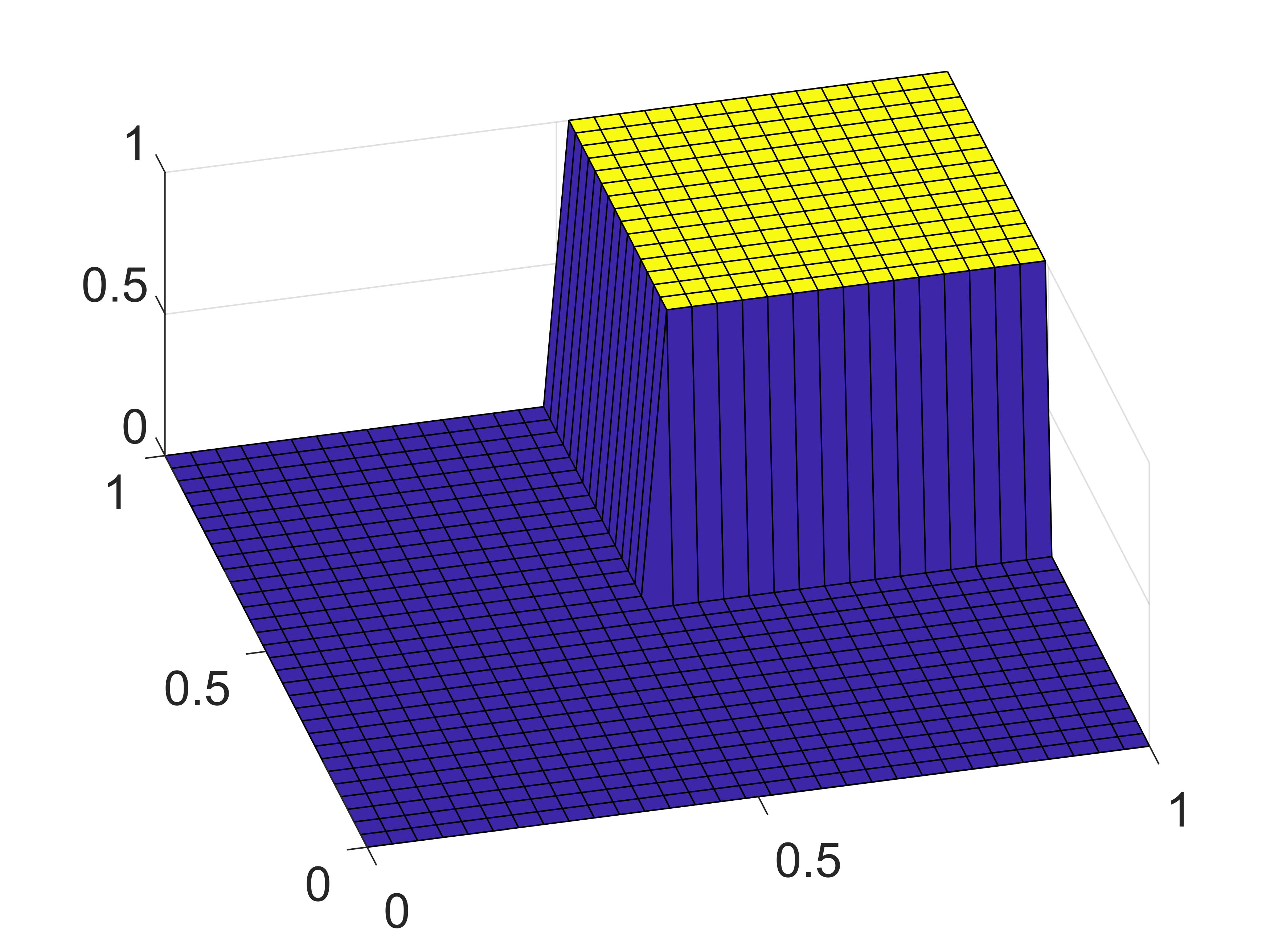}
\includegraphics[width=0.32\textwidth]{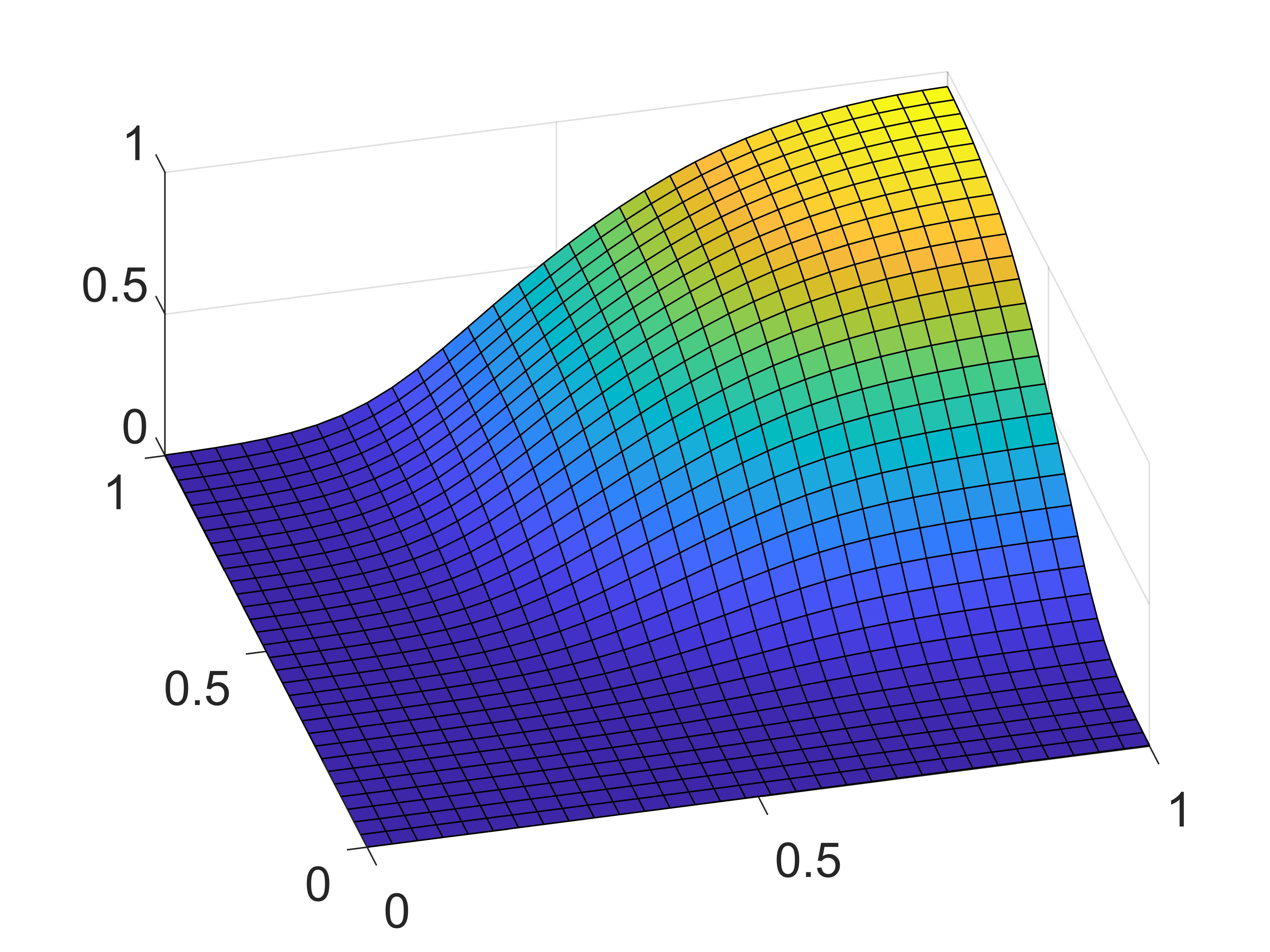}
\includegraphics[width=0.32\textwidth]{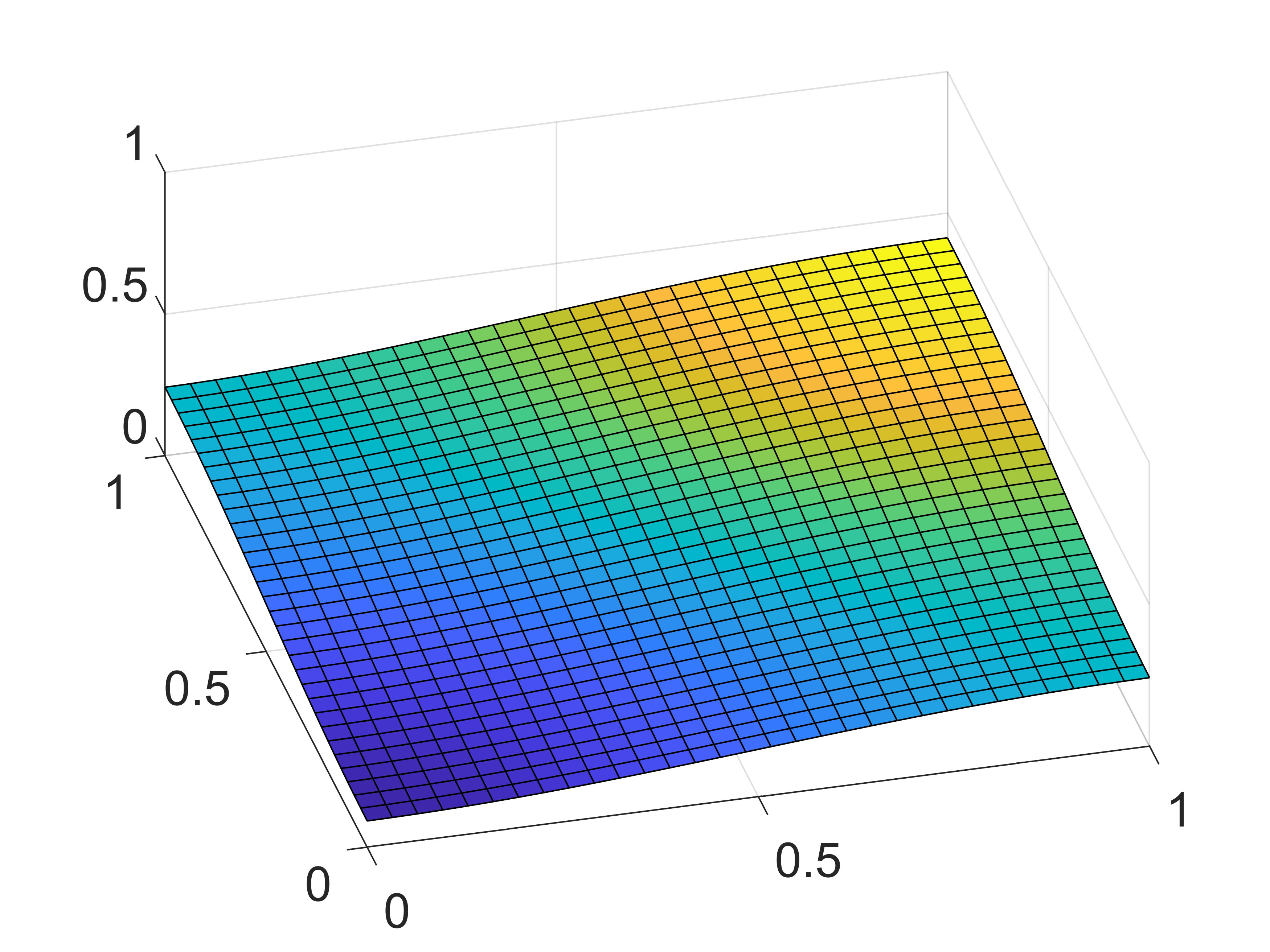}
\caption{Density $u_1(t)$ (upper row) and $u_2(t)$ (lower row) at times
$t=0,0.02,0.2$ (from left to right) versus space.}
\label{fig.two}
\end{figure}

\subsection{Third example: exponential time decay}

We choose the one-dimensional domain $\Omega=(0,1)$,
$\gamma=0.1$, $\Delta x=2^{-7}$, $\Delta t=(10\cdot 2^7)^{-1}$, 
and 
$$
  A = \begin{pmatrix} \beta & 2 \\ 2 & 1 \end{pmatrix}, \quad
  u^0(x) = \begin{pmatrix} 2-\cos(\pi x) \\ 2+\cos(\pi x) \end{pmatrix},
$$
where $\beta>4$. 
The distance $\|A^{1/2}(u^k-\bar{u})\|_{L^2(\Omega)}$ presented in Figure \ref{fig.three} for $\beta=5$ and $\beta=4.01$ shows that the time decay behaves exponentially, as predicted by Theorem \ref{thm.large}.
The decay rates (excluding the initial decay) 
are $-4.37$ for $\beta=5$ and $-1.03$ 
for $\beta=4.01$, and they decrease for smaller values of $\det A$.
We have also observed an exponential decay when $\gamma=0$
with smaller decay rates.

\begin{figure}[ht]
\includegraphics[width=0.49\textwidth]{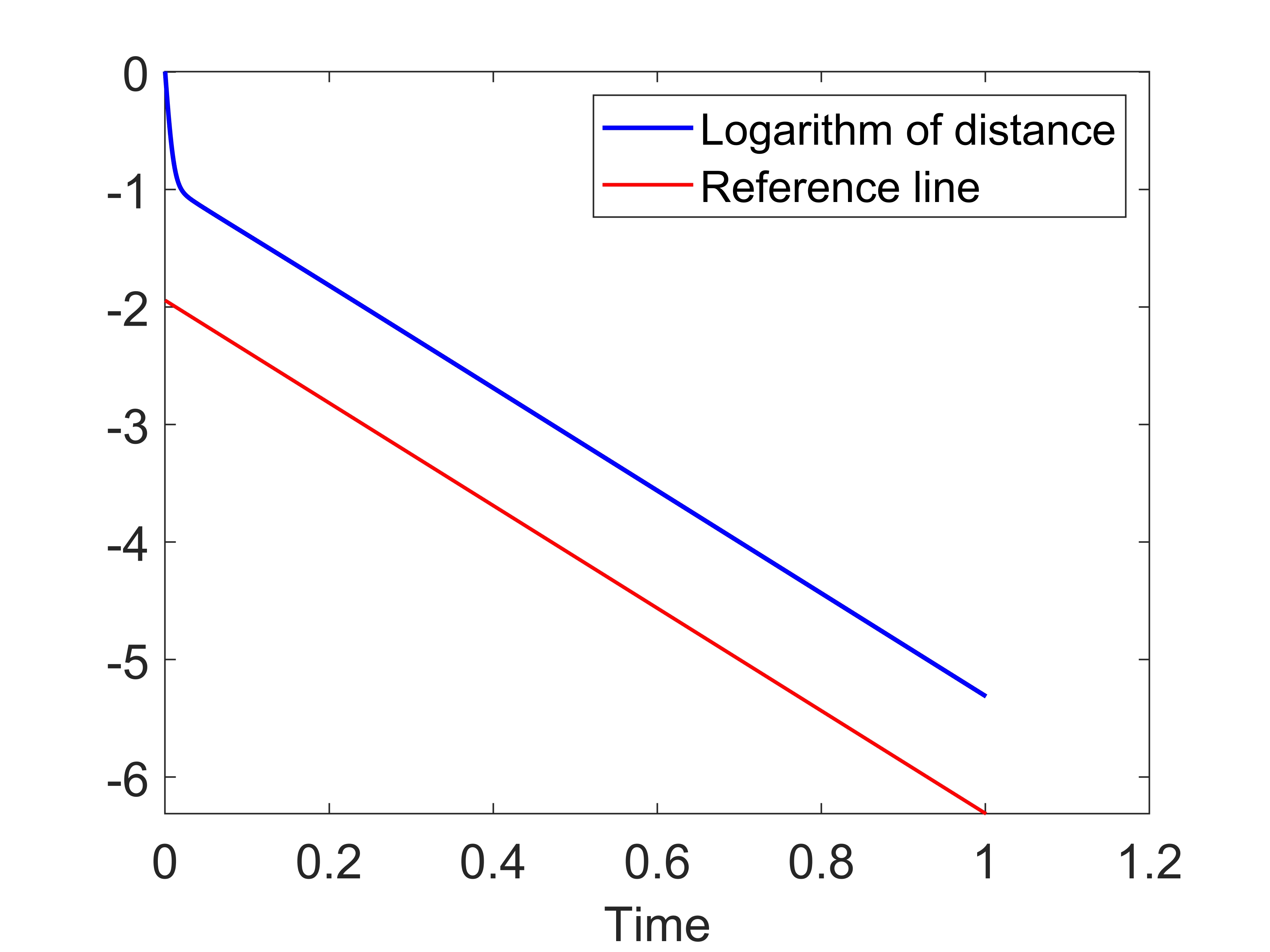}
\includegraphics[width=0.49\textwidth]{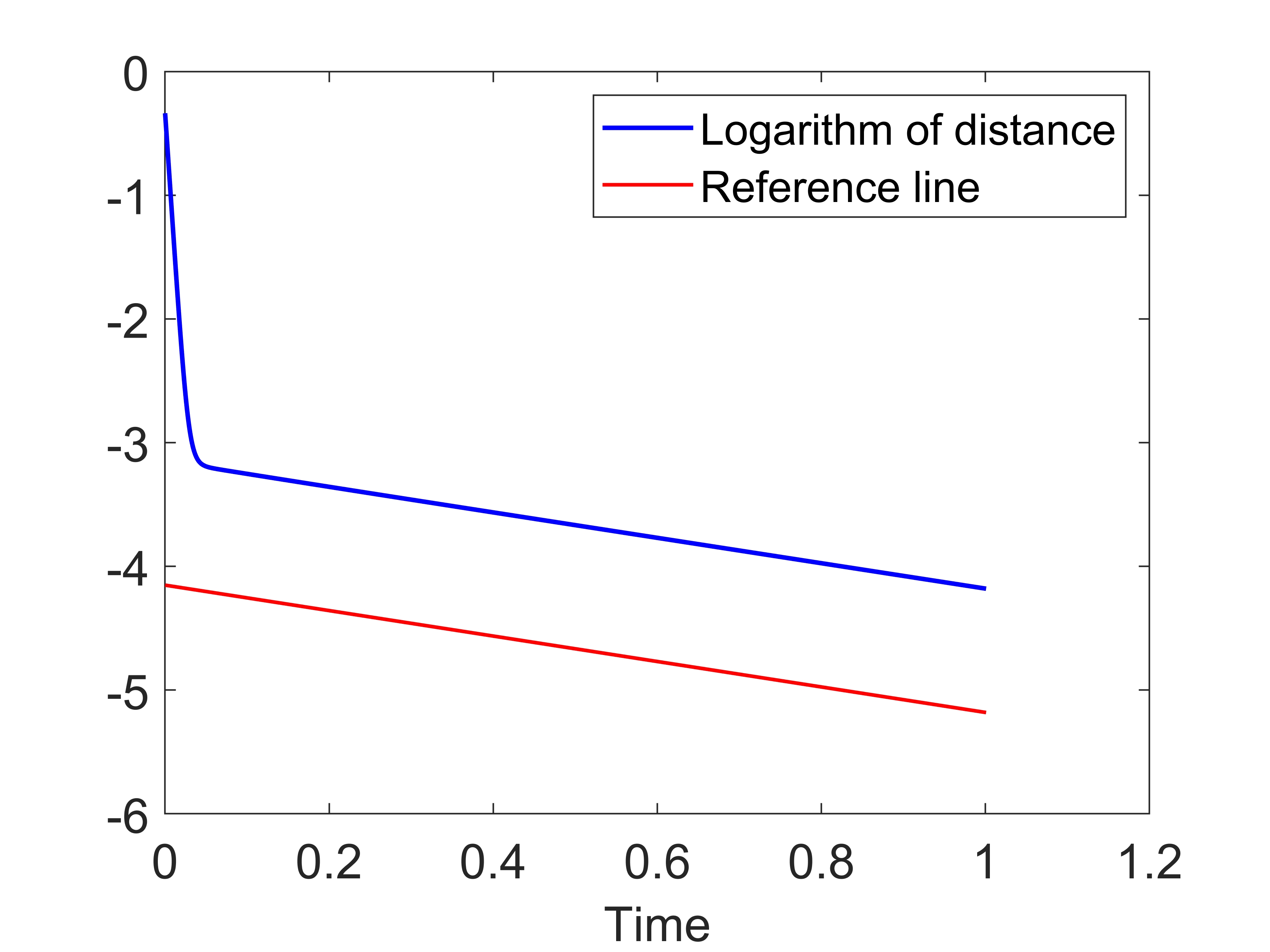}
\caption{Semilogarithmic plot for $\|A^{1/2}(u^k-\bar{u})\|_{L^2(\Omega)}$ 
versus time $t_k$.}
\label{fig.three}
\end{figure}

\subsection{Fourth example: Convergence rate in time}

We choose the values for $A$ and $u^0$ as in the previous example as well as
$\gamma=0$, $\Delta x=2^{-9}$, and $\Delta t=(10\cdot 2^p)^{-1}$ with $p=1,\ldots,8$.
The reference solution $u_{\rm ref}$
is computed with the time step size $\Delta t=(10\cdot 2^9)^{-1}$.
As expected, the convergence rate at time $T=0.02$, shown in Figure \ref{fig.four} for two different values of $\beta$, 
is about two, even in the case $\det A=0$.

\begin{figure}[ht]
\includegraphics[width=0.49\textwidth]{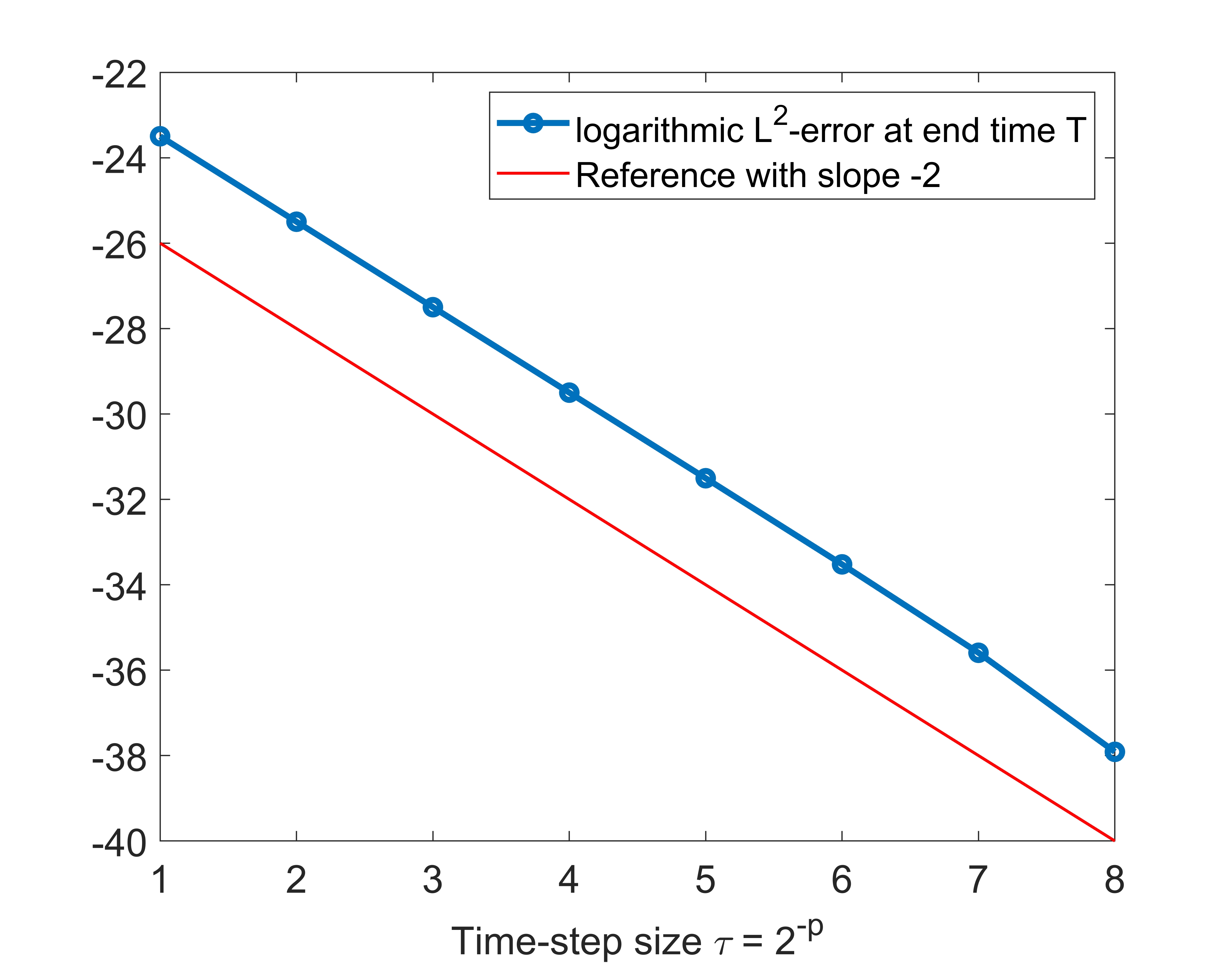}
\includegraphics[width=0.49\textwidth]{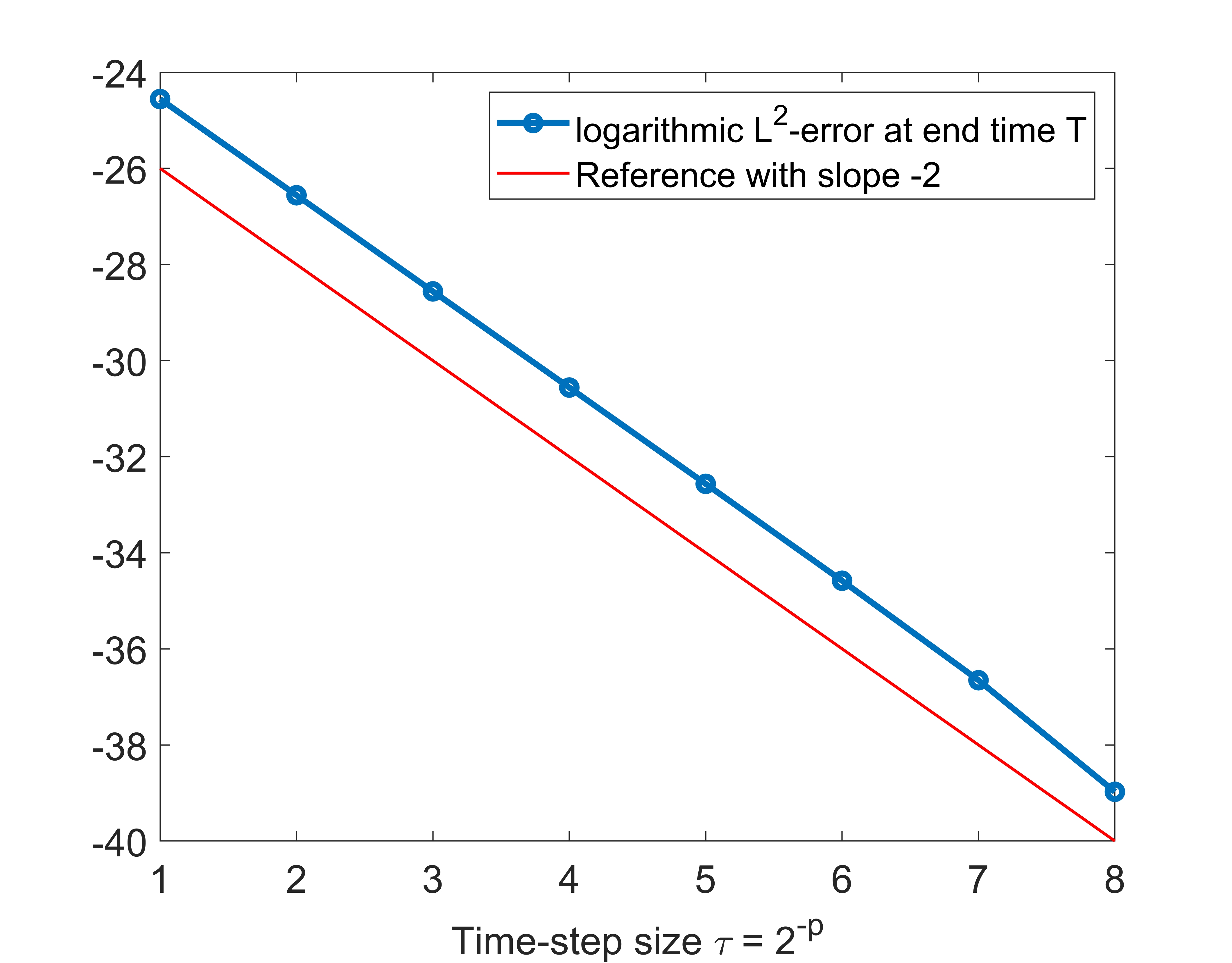}
\caption{Discrete $L^2(\Omega)$ error 
$\|A^{1/2}(u^{(\Delta t)}-u_{\rm ref})(T)\|_{L^2(\Omega)}$
versus time step size $\Delta t=(10\cdot 2^{p})^{-1}$ for $p=1,\ldots,8$
for $\beta=5$ (left) and $\beta=4$ (right).}
\label{fig.four}
\end{figure}



\begin{thebibliography}{11}
\bibitem{Ama93} H.~Amann. Nonhomogeneous linear and quasilinear elliptic and 
parabolic boundary value problems. In: H.~J.~Schmeisser and H.~Triebel (eds.), 
Funct. Spaces Differ. Op. Nonlin. Anal., pp. 9--126. Teubner, Wiesbaden, 1993.

\bibitem{ABR11} B.~Andreianov, M.~Bendahmane, and R.~Ruiz-Baier. Analysis of a finite-volume
method for a cross-diffusion model in population dynamics. {\em Math. Models Meth. Appl. Sci.}
21 (2011), 307--344.

\bibitem{BCH20} R.~Bailo, J.~A.~Carrillo, and J.~Hu. Fully discrete positivity-preserving and 
energy-dissipating schemes for aggregation-diffusion equations with a gradient-flow structure.
{\em Commun. Math. Sci.} 18 (2020), 1259--1303.

\bibitem{BGHP85} M.~Bertsch, M.~Gurtin, D.~Hilhorst, and L.~Peletier. On interacting 
populations that disperse to avoid crowding: preservation of segregation.
{\em J. Math. Biol.} 23 (1985), 1--13.

\bibitem{BHIM12} M.~Bertsch, D.~Hilhorst, H.~Izuhara, and M.~Mimura. A nonlinear
parabolic-hyperbolic system for contact inhibition of cell-growth.
{\em Differ. Eqs. Appl.} 4 (2012), 137--157.

\bibitem{Bes12} M.~Bessemoulin-Chatard. A finite volume scheme for convection-diffusion equations with nonlinear diffusion derived from the Scharfetter--Gummel scheme. {\em Numer. Math.} 121 (2012), 637--670.

\bibitem{BCF15} M.~Bessemoulin-Chatard, C.~Chainais-Hillairet, and F.~Filbet.
On discrete functional inequalities for some finite volume schemes.
{\em IMA J. Numer. Anal.} 35 (2015), 1125--1149.

\bibitem{CaEz17} C.~Calgaro and M.~Ezzoug. $L^\infty$-stability of the IMEX-BDF2 finite 
volume scheme for convection-diffusion equation. In: C.~Canc\`es and P.~Omnes (eds.), 
{\em Finite Volumes for Complex Applications VIII}, pp.~245--253. Springer, Cham, 2017.

\bibitem{CaGa20} C.~Canc\`es and B.~Gaudeul. A convergent entropy diminishing finite volume 
scheme for a cross-diffusion system. {\em SIAM J. Numer. Anal.} 58 (2020), 2684--2710.

\bibitem{CLP03} C.~Chainais-Hillairet, J.-G.~Liu, and Y.-J.~Peng. Finite volume 
scheme for multi-dimensional drift-diffusion equations and convergence analysis. 
{\em ESAIM Math. Model. Numer. Anal.} 37 (2003), 319--338.

\bibitem{CDJ19} L.~Chen, E.~Daus, and A.~J\"ungel. Rigorous mean-field limit and cross diffusion. {\em Z. Angew. Math. Phys.} 70 (2019), no.~122, 21 pages. 

\bibitem{ChJu19} X.~Chen and A.~J\"ungel. Weak-strong uniqueness of renormalized solutions to reaction-cross-diffusion systems.
{\em Math. Models Meth. Appl. Sci.} 29 (2019), 237--270.

\bibitem{CWWW19} W.~Chen, C.~Wang, X.~Wang, and S.~Wise. Positivity-preserving, energy stable 
numerical schemes for the Cahn--Hilliard equation with logarithmic potential.
{\em J. Comput. Phys.} X 3 (2019), no.~1000031, 29 pages.

\bibitem{DFGH03} W.~Dahmen, B.~Faermann, I.~Graham, W.~Hackbusch, and 
S.~Sauter. Inverse inequalities on non-quasi-uniform meshes and 
applications to the Mortar element method. 
{\em Math. Comput.} 73 (2004), 1107--1138.

\bibitem{DWZZ20} L.~Dong, C.~Wang, H.~Zhang, and Z.~Zhang. A positivity-preserving 
second-order BDF scheme for the Cahn--Hilliard equation with variable interfacial parameters.
{\em Commun. Comput. Phys.} 28 (2020), 967--998.

\bibitem{DrJu12} M.~Dreher and A.~J\"ungel. Compact families of piecewise constant functions 
in $L^p(0,T;B)$. {\em Nonlin. Anal.} 75 (2012), 3072--3077.

\bibitem{DrNa18} J.~Droniou and N.~Nataraj. Improved $L^2$ estimate for gradient schemes and 
super-convergence of the TPFA finite volume scheme. 
{\em IMA J. Numer. Anal.} 38 (2018), 1254--1293.

\bibitem{DHJ22} P.-E.~Druet, K.~Hopf, and A.~J\"ungel. Hyperbolic-parabolic
normal form and local classical solutions for cross-diffusion systems with 
incomplete diffusion. Submitted for publication, 2022. arXiv:2210.17244.

\bibitem{Emm09} E.~Emmrich. Two-step BDF time discretization of nonlinear evolution problems 
governed by monotone operators with strongly continuous perturbations. 
{\em Comput. Meth. Appl. Math.} 9 (2009), 37--62.

\bibitem{EGH99} R.~Eymard, T.~Gallou{\"e}t, and R.~Herbin. Convergence of 
finite volume schemes for semilinear convection diffusion equations.
{\em Numer. Math.} 82 (1999), 91--116.

\bibitem{EGH00} R.~Eymard, T.~Gallou{\"e}t, and R.~Herbin. Finite volume methods. 
In: P.~G.~Ciarlet and J.-L.~Lions (eds.).
{\em Handbook of Numerical Analysis} 7 (2000), 713--1018.

\bibitem{GuSh20} Y.~Gu and J.~Shen. Bound preserving and energy dissipative schemes for 
porous medium equation. {\em J. Comput. Phys.} 410 (2020), no.~109378, 21 pages.

\bibitem{Hil97} A.~Hill. Global dissipativity for A-stable methods. 
{\em SIAM J. Numer. Anal.} 34 (1997), 119--142.


\bibitem{JuMi15} A.~J\"ungel and J.-P.~Mili\v{s}i\`c. Entropy dissipative one-leg multistep time 
approximations of nonlinear diffusive equations. 
{\em Numer. Meth. Part. Diff. Eqs.} 31 (2015), 1119--1149.  

\bibitem{JPZ22} A.~J\"ungel, S.~Portisch, and A.~Zurek. Nonlocal cross-diffusion systems 
for multi-species populations and networks. {\em Nonlin. Anal.} 219 (2022), no.~112800,
26 pages.

\bibitem{JuZu20} A.~J\"ungel and A.~Zurek. A finite-volume scheme for a cross-diffusion
model arising from interacting many-particle population systems.
In: R.~Kl\"ofkorn, E.~Keilegavlen, F.~Radu, and J.~Fuhrmann (eds.), 
{\em Finite Volumes for Complex Applications IX}, pp.~223--231. Springer, Cham, 2020. 

\bibitem{JuZu22} A.~J\"ungel and A.~Zurek. A discrete boundedness-by-entropy method for 
finite-volume appro\-xi\-ma\-tions of cross-diffusion systems. To appear in
{\em IMA J. Math. Anal.}, 2022. \newline https://doi.org/10.1093/imanum/drab101.

\bibitem{Kri15} R.~Krishna. Uphill diffusion in multicomponent mixtures.
{\em Chem. Soc. Rev.} 44 (2015), 2812--2836.

\bibitem{MaPl19} D.~Matthes and S.~Plazetta. A variational formulation of the BDF2 method 
for metric gradient flows. {\em ESAIM Math. Model. Numer. Anal.} 53 (2019), 145--172. 

\bibitem{Rao82} C.~Rao. Diversity and dissimilarity coefficients: a unified approach. 
{\em Theor. Popul. Biol.} 21 (1982), 24--43.

\end{thebibliography}
\end{document}